\documentclass[11pt,reqno]{article}

\usepackage{float}
\restylefloat{table*}
\usepackage{booktabs} 

\usepackage{amsmath, amssymb} 
\usepackage{empheq}
\usepackage{theorem}	
\usepackage{mathpazo}
\usepackage[margin=3cm]{geometry}
\usepackage[title]{appendix}

\usepackage{caption,subcaption}

\usepackage{graphicx, xcolor, soul}
\definecolor{labelkey}{rgb}{0,0.08,0.45}
\definecolor{refkey}{rgb}{0,0.6,0.0}
\definecolor{Brown}{rgb}{0.45,0.0,0.05}
\definecolor{lime}{rgb}{0.00,0.8,0.0}
\definecolor{lblue}{rgb}{0.5,0.5,0.99}

\usepackage{hyperref}

\definecolor{myblue}{rgb}{.9, .9, 1}
  \newcommand*\mybluebox[1]{%
    \colorbox{myblue}{\hspace{1em}#1\hspace{1em}}}

\newcommand{\sepp}{\setlength{\itemsep}{-2pt}}

\newcommand{\menge}[2]{\left\{{#1}~\big |~{#2}\right\}} 
\newcommand{\mmenge}[2]{\bigg\{{#1}~\bigg |~{#2}\bigg\}}

\newcommand{\ball}[2]{\operatorname{ball}({#1};{#2})}

\newcommand{\scal}[2]{\left\langle {#1},{#2} \right\rangle}

\newcommand{\To}{\ensuremath{\rightrightarrows}}

\newcommand{\ve}{\ensuremath{\varepsilon}}
\newcommand{\exi}{\ensuremath{\exists\,}}

\newcommand{\NN}{\ensuremath{\mathbb N}}
\newcommand{\nnn}{\ensuremath{{n\in{\mathbb N}}}}

\newcommand{\RR}{\ensuremath{\mathbb R}}
\newcommand{\RP}{\ensuremath{\mathbb{R}_+}}
\newcommand{\RPP}{\ensuremath{\mathbb{R}_{++}}}

\newcommand{\ba}{\ensuremath{\mathbf{a}}}
\newcommand{\bb}{\ensuremath{\mathbf{b}}}

\newcommand{\bx}{\ensuremath{\mathbf{x}}}

\newcommand{\bz}{\ensuremath{\mathbf{z}}}

\newcommand{\bA}{\ensuremath{{\mathbf{A}}}}
\newcommand{\bB}{\ensuremath{{\mathbf{B}}}}

\newcommand{\bT}{\ensuremath{{\mathbf{T}}}}
\newcommand{\bX}{\ensuremath{{\mathbf{X}}}}

\newcommand{\bd}{\ensuremath{\operatorname{bdry}}}
\newcommand{\inte}{\ensuremath{\operatorname{int}}}

\newcommand{\reli}{\ensuremath{\operatorname{ri}}}

\newcommand{\aff}{\ensuremath{\operatorname{aff}}}
\newcommand{\lspan}{\ensuremath{{\operatorname{span}}\,}}

\newcommand{\epi}{\ensuremath{\operatorname{epi}}}
\newcommand{\gra}{\ensuremath{\operatorname{gra}}}
\newcommand{\Fix}{\ensuremath{\operatorname{Fix}}}
\newcommand{\Id}{\ensuremath{\operatorname{Id}}}

\newtheorem{theorem}{Theorem}[section]
\newtheorem{lemma}[theorem]{Lemma}
\newtheorem{corollary}[theorem]{Corollary}
\newtheorem{proposition}[theorem]{Proposition}
\newtheorem{definition}[theorem]{Definition}
\theoremstyle{plain}{\theorembodyfont{\rmfamily}
}
\theoremstyle{plain}{\theorembodyfont{\rmfamily}
}
\theoremstyle{plain}{\theorembodyfont{\rmfamily}
}
\theoremstyle{plain}{\theorembodyfont{\rmfamily}
\newtheorem{example}[theorem]{Example}}
\newtheorem{fact}[theorem]{Fact}
\theoremstyle{plain}{\theorembodyfont{\rmfamily}
\newtheorem{remark}[theorem]{Remark}}

\newenvironment{proof}[1][Proof]{\par
	\vspace{-12pt}
	\trivlist
	\item[\hskip\labelsep\itshape #1.]\ignorespaces
	}{%
	\hfill\ensuremath{\blacksquare}\endtrivlist
}

\allowdisplaybreaks

\begin{document}

\title{On Slater's condition and finite convergence of the
Douglas--Rachford algorithm}

\author{
Heinz H.\ Bauschke\thanks{
Mathematics, University of British Columbia, Kelowna, B.C.\ V1V~1V7, Canada. 
E-mail: \texttt{heinz.bauschke@ubc.ca}.},~
Minh N.\ Dao\thanks{
Department of Mathematics and Informatics, Hanoi National University of Education, 136 Xuan Thuy, Hanoi, Vietnam,
and Mathematics, University of British Columbia, Kelowna, B.C.\ V1V~1V7, Canada.
E-mail: \texttt{minhdn@hnue.edu.vn}.},~
Dominikus Noll\thanks{
Institut de Math{\'e}matiques, Universit{\'e} de Toulouse, 118 route de Narbonne, 31062 Toulouse, France.
E-mail: \texttt{noll@mip.ups-tlse.fr}.}~~~and~
Hung M.\ Phan\thanks{
Department of Mathematical Sciences, University of Massachusetts Lowell,
265 Riverside St., Olney Hall 428, Lowell, MA 01854, USA. 
E-mail: \texttt{hung\_phan@uml.edu}.}
}

\date{April 27, 2015}

\maketitle

\begin{abstract} \noindent
The Douglas--Rachford algorithm is a classical and very
successful method for solving optimization and feasibility
problems. In this paper, we provide novel conditions sufficient for
finite convergence in the context of convex feasibility problems. 
Our analysis builds upon, and considerably extends, pioneering
work by Spingarn. Specifically, we obtain finite convergence in
the presence of Slater's condition in the affine-polyhedral and
in a hyperplanar-epigraphical case. Various examples illustrate
our results. Numerical experiments demonstrate the
competitiveness of the Douglas--Rachford algorithm for solving linear equations with a
positivity constraint when compared to the method of alternating
projections and the method of reflection-projection. 
\end{abstract}

{\small

}

\section{Introduction}

Throughout this paper, we assume that 
\begin{empheq}[box=\mybluebox]{equation}
\text{$X$ is a finite-dimensional real Hilbert space}
\end{empheq}
with inner product $\scal{\cdot}{\cdot}$ and induced norm $\|\cdot\|$, and 
\begin{empheq}[box=\mybluebox]{equation}
\text{$A$ and $B$ are closed convex subsets of $X$ such that $A\cap B\neq\varnothing$}.
\end{empheq}
Consider the \emph{convex feasibility problem}
\begin{empheq}[box=\mybluebox]{equation}
\label{e:prob}
\text{find a point in $A\cap B$}
\end{empheq}
and assume that it is possible to evaluate the \emph{projectors} (nearest point mappings) 
$P_A$ and $P_B$ corresponding to $A$ and $B$, respectively. 
We denote the corresponding reflectors by
$R_A := 2P_A-\Id$ and $R_B := 2P_B-\Id$, respectively. 
\emph{Projection methods} combine the projectors and reflectors
in a suitable way to generate a sequence converging to a solution
of \eqref{e:prob} --- we refer the reader to 
\cite{BC11}, \cite{Ceg12}, and \cite{CZ97} and the references
therein for further information. 

One celebrated algorithm for solving \eqref{e:prob}
is the so-called \emph{Douglas--Rachford Algorithm (DRA)}
\cite{DR56}. The adaption of this algorithm to
optimization and feasibility is actually due to Lions and
Mercier was laid out beautifully in their landmark paper
\cite{LM79} (see also \cite{EB92}). 
The DRA is based on the Douglas--Rachford splitting operator, 
\begin{empheq}[box=\mybluebox]{equation}
T :=\Id -P_A +P_BR_A,
\end{empheq}
which is used to generate a sequence $(z_n)_\nnn$ with 
starting point $z_0\in X$ via
\begin{empheq}[box=\mybluebox]{equation}
(\forall\nnn) \quad z_{n+1} := Tz_n.
\end{empheq}
Then the ``governing sequence'' $(z_n)_\nnn$ converges to a point $z\in \Fix T$,
and, more importantly, the ``shadow sequence'' $(P_Az_n)_\nnn$
converges to $P_Az$ which is a solution of \eqref{e:prob}. 

An important question concerns the speed of convergence of the
sequence $(P_Az_n)_\nnn$. 
Linear convergence was more clearly understood recently,
see \cite{HL13}, \cite{Pha14}, and \cite{BNP15}. 

\emph{The aim of this paper is to provide \emph{verifiable} conditions
sufficient for \emph{finite} convergence.}

Our two main results reveal that \emph{Slater's condition}, i.e.,
\begin{equation}
A\cap \inte B \neq\varnothing
\end{equation}
plays a key role and guarantees finite convergence when
\textbf{(MR1)} $A$ is an affine subspace and $B$ is a polyhedron
(Theorem~\ref{t:Bpoly}); or when
\textbf{(MR2)} $A$ is a certain hyperplane and $B$ is an epigraph 
(Theorem~\ref{t:epi}). 
Examples illustrate that these results are applicable in
situations where previously known conditions sufficient for
finite convergence fail.
When specialized to a product space setting, we derive
a finite convergence result due to Spingarn \cite{Spi85} for his
method of partial inverses \cite{Spi83}. Indeed, the proof of
Theorem~\ref{t:Bpoly} follows his pioneering work, but, at the same
time, we simplify his proofs and strengthen the conclusions.
These sharpenings allow us to obtain finite-convergence results
for solving linear equations with a positivity constraint. 
Numerical experiments support the competitiveness of the DRA for
solving \eqref{e:prob}. 

\subsection*{Organization of the paper}
The paper is organized as follows.
In Section~\ref{s:aux}, we present
several auxiliary results which make the eventual proofs
of the main results more structured and transparent. 
Section~\ref{s:main1} contains the first main result
\textbf{(MR1)}.
Applications using the product space set up, a comparison
with Spingarn's work, and numerical experiments are provided in Section~\ref{s:four}. 
The final Section~\ref{s:hypepi} concerns the second main result
\textbf{(MR2)}. 

\subsection*{Notation}
The notation employed is standard and follows largely \cite{BC11}.
The real numbers are $\RR$, and the nonnegative integers are $\NN$.
Further, $\RP := \menge{x \in \RR}{x \geq 0}$ and $\RPP := \menge{x \in \RR}{x >0}$.
Let $C$ be a subset of $X$. 
Then the closure of $C$ is $\overline C$, the interior of $C$ is $\inte C$, the boundary of $C$ is $\bd C$,
and the smallest affine and linear subspaces containing $C$ are,
respectively, $\aff C$ and $\lspan C$. 
The relative interior of $C$, $\reli C$, is the interior of $C$ relative to $\aff C$.
The orthogonal complement of $C$ is $C^\perp :=\menge{y \in X}{(\forall x\in C)\; \scal{x}{y} =0}$,
and the dual cone of $C$ is $C^\oplus :=\menge{y \in X}{(\forall x \in C)\; \scal{x}{y} \geq 0}$.
The normal cone operator of $C$ is denoted by $N_C$, i.e.,
$N_C(x) =\menge{y \in X}{(\forall c \in C)\; \scal{y}{c-x} \leq 0}$ if $x \in C$, and $N_C(x) =\varnothing$ otherwise.
If $x \in X$ and $\rho \in \RPP$, 
then $\ball{x}{\rho} :=\menge{y \in X}{\|x -y\| \leq \rho}$ 
is the closed ball centered at $x$ with radius $\rho$.

\section{Auxiliary results}
\label{s:aux}
In this section, we collect several auxiliary results that will be useful in the sequel. 

\subsection{Convex sets}

\begin{lemma}
\label{l:Ncone}
Let $C$ be a nonempty closed convex subset of $X$, let $x \in X$, 
and let $y \in C$. 
Then 
$x -P_Cx \in N_C(y)$ $\Leftrightarrow$ $\scal{x -P_Cx}{y -P_Cx} =0$.
\end{lemma}
\begin{proof}
Because 
$y\in C$ and $x-P_Cx\in N_C(P_Cx)$, we have $\scal{x-P_Cx}{y-P_Cx}\leq 0$. 
``$\Rightarrow$'':
From $x-P_Cx\in N_C(y)$ and $P_Cx \in C$,
we have $\scal{x-P_Cx}{P_Cx-y}\leq 0$. 
Thus $\scal{x-P_Cx}{P_Cx-y}=0$. 
``$\Leftarrow$'':
We have $(\forall c\in C)$ 
$\scal{x-P_Cx}{c-y} =
\scal{x-P_Cx}{c-P_Cx}+\scal{x-P_Cx}{P_Cx-y} =
\scal{x-P_Cx}{c-P_Cx} \leq 0$.
Hence $x-P_Cx\in N_C(y)$. 
\end{proof}

\begin{lemma}
\label{l:int}
Let $C$ be a nonempty convex subset of $X$. Then 
$\inte C \neq\varnothing $
$\Leftrightarrow$
$0 \in \inte(C-C)$.
\end{lemma}
\begin{proof}
``$\Rightarrow$'': Clear. 
''$\Leftarrow$'': 
By \cite[Theorem~6.2]{Roc70}, $\reli C \neq 0$. 
After translating the set if necessary, we assume that 
$0 \in \reli C$. 
Then $0 \in C$, and so $Y :=\aff C =\lspan C$. 
Since $0 \in \inte(C -C)\subseteq \inte(\lspan C)=\inte Y$, 
this gives $\inte Y \neq\varnothing$ and thus 
$Y =X$. In turn, $\inte C =\reli C \neq\varnothing$.
\end{proof}

\subsection{Cones} 

\begin{lemma}
\label{l:cvxcone}
Let $K$ be a nonempty convex cone in $X$. 
Then there exists $v \in \reli K \cap \reli K^\oplus$ such that
\begin{equation}
(\forall x \in \overline K\smallsetminus(\overline K\cap(-\overline K))) 
\quad \scal{v}{x} >0.
\end{equation}
\end{lemma}
\begin{proof}
By \cite[Lemma~2]{Spi85}, there exists $v \in \reli K \cap \reli K^\oplus$.
Then $v \in \reli K^\oplus =\reli \overline K^\oplus$. 
Noting that $\overline K$ is a closed convex cone and 
using \cite[Lemma~3]{Spi85}, we complete the proof.
\end{proof}

\begin{lemma}
\label{l:plfun}
Let $(z_n)_\nnn$ be a sequence in $X$, 
and let $f \colon X \to \RR$ be linear. 
Assume that $z_n \to z \in X$, and that
\begin{equation}
\label{e:decrease}
(\forall\nnn)\quad f(z_n) >f(z_{n+1}).
\end{equation}
Then there exist $n_0 \in \NN\smallsetminus \{0\}$ and 
$(\mu_1, \dots, \mu_{n_0}) \in \RP^{n_0}$ such that 
$\sum_{k=1}^{n_0}\mu_k=1$ and 
\begin{equation}
\label{e:pscal}
\scal{\sum_{k=1}^{n_0} \mu_k(z_{k-1} -z_k)}{\sum_{k=1}^{n_0} \mu_k(z_k -z)} >0.
\end{equation}
\end{lemma}
\begin{proof}
Introducing
\begin{equation}
(\forall n \in \NN\smallsetminus \{0\})\quad y_n :=z_{n-1} -z_n,
\end{equation}
we get
\begin{equation}
(\forall n \in \NN\smallsetminus \{0\})\quad f(y_n) >0, \quad\text{and so}\quad y_n \neq 0.
\end{equation} 
Let $K$ be the convex cone generated by 
$\{y_n\}_{n \in \NN\smallsetminus \{0\}}$. 
We see that
\begin{equation}
(\forall x \in K\smallsetminus\{0\})\quad f(x) >0, 
\quad\text{and}\quad 
(\forall x \in -K\smallsetminus\{0\})\quad f(x) <0.
\end{equation}
Therefore, 
\begin{equation}
\label{e:linC}
\overline K\cap (-\overline K) \subseteq \menge{x \in X}{f(x) =0}. 
\end{equation}
Setting 
\begin{equation}
(\forall n \in \NN\smallsetminus \{0\})\quad w_n: =z_n -z,
\end{equation}
we immediately have $w_n \to 0$, and so 
\begin{equation}
\label{e:znC}
(\forall n \in \NN\smallsetminus \{0\})\quad w_n =z_n -z =\sum_{k=n+1}^\infty (z_{k-1} -z_k) =\sum_{k=n+1}^\infty y_k \in \overline K.
\end{equation}
From \eqref{e:decrease} we get
\begin{equation}
(\forall n \in \NN\smallsetminus \{0\})\quad f(w_n) >f(w_{n+1}). 
\end{equation}
Moreover, $f(w_n) \to f(0) =0$, hence  
\begin{equation}
(\forall n \in \NN\smallsetminus \{0\})\quad f(w_n) >0.
\end{equation}
Together with \eqref{e:linC} and \eqref{e:znC}, this gives 
\begin{equation}
(\forall n \in \NN\smallsetminus \{0\})\quad w_n \in \overline K\smallsetminus(\overline K\cap(-\overline K)).
\end{equation}
By Lemma~\ref{l:cvxcone}, 
there exists $v \in \reli K\cap \reli K^\oplus$ such that
\begin{equation}
\label{e:vwn}
(\forall n \in \NN\smallsetminus \{0\})\quad \scal{v}{w_n} >0.
\end{equation}
Then we must have $v \neq 0$. Since $v \in K$, after scaling if necessary,
there exist $n_0\in\NN\smallsetminus\{0\}$ and 
$(\mu_1,\ldots,\mu_{n_0})\in\RP^{n_0}$ such that 
\begin{equation}
v =\sum_{k=1}^{n_0} \mu_k y_k, \quad \text{and}\quad \sum_{k=1}^{n_0} \mu_k =1.
\end{equation} 
This combined with \eqref{e:vwn} implies
\begin{equation}
\scal{\sum_{k=1}^{n_0} \mu_k y_k}{\sum_{k=1}^{n_0} \mu_k w_k} >0,
\end{equation}
and so \eqref{e:pscal} holds.
\end{proof}

\begin{lemma}
\label{l:pointed}
Let $K$ be a nonempty pointed\footnote{Recall that a cone $K$ is
pointed if $K\cap (-K)\subseteq \{0\}$.} convex cone in $X$. 
Then the following hold:
\begin{enumerate}
\item 
\label{l:pointed1}
Let $m\in\NN\smallsetminus\{0\}$ and let $(x_1,\ldots,x_m)\in
K^m$. Then 
$x_1 +\cdots +x_m =0$
$\Leftrightarrow$ $x_1 =\cdots =x_m =0$.
\item
\label{l:pointed2}
If $K$ is closed and $L\colon X\to X$ is linear such that 
\begin{equation}
\label{e:kerL}
\ker L\cap K =\{0\}, 
\end{equation}
then $L(K)$ is a nonempty pointed closed convex cone.
\end{enumerate}
\end{lemma}
\begin{proof} \ 
\ref{l:pointed1}: 
Assume that $x_1 +\cdots +x_m =0$. Then since $K$ is a convex cone, 
\begin{equation}
-x_1 =x_2 +\cdots +x_m \in K,
\end{equation}
and so $x_1 \in K\cap (-K)$. 
Since $K$ is pointed, we get $x_1 =0$. 
Continuing in this fashion, we eventually conclude that 
$x_1 =\cdots =x_m =0$. The converse is trivial. 

\ref{l:pointed2}: Since $K$ is a closed convex cone, so is $M :=L(K)$ 
due to assumption \eqref{e:kerL} and \cite[Proposition~3.4]{BM09}.
Now let $z \in M\cap (-M)$. Then $z =L(r) =-L(s)$ for some points 
$r,s$ in $K$.
Thus $L(r +s) =L(r) +L(s) =0$, which gives $r +s \in \ker L$, and so $r +s \in \ker L\cap K$. 
By again \eqref{e:kerL}, $r +s =0$, and now \ref{l:pointed1} implies $r =s =0$. 
Therefore, $z =0$, and $M$ is pointed. 
\end{proof}

\begin{lemma}
\label{l:stop}
Let $(a_n)_\nnn$ be a sequence in $X$ such that $a_n \to a \in X$, and $K$ be a pointed closed convex cone of $X$.
Assume that 
\begin{equation}
\label{e:aaK}
(\exi p\in\NN)\quad
a_p=a\;\;\text{and}\;\;
(\forall n \geq p)\quad a_n -a_{n+1} \in K.
\end{equation}
Then
\begin{equation}
\label{e:stop}
(\forall n \geq p)\quad a_n =a.
\end{equation}
\end{lemma}
\begin{proof} 
Since $K$ is a closed convex cone and $a_n \to a$, it follows from \eqref{e:aaK} that
\begin{equation}
(\forall n \geq p)\quad a_n -a =\sum_{k=n}^\infty (a_k  -a_{k+1}) \in K,
\end{equation}
and so $a_{p+1} -a \in K$. 
Since $a_p =a$, \eqref{e:aaK} gives $a -a_{p+1} \in K$.  
Noting that $K$ is pointed, this implies $a_{p+1} -a \in K\cap
(-K) \subseteq \{0\}$, and hence $a_{p+1} =a$. 
Repeating this argument, we get the conclusion. 
\end{proof}

\subsection{Locally polyhedral sets}

\begin{definition}[local polyhedrality]
Let $C$ be a subset of $X$.
We say that $C$ is \emph{polyhedral at} $c\in C$ if there
exist a polyhedral\footnote{Recall that a set is polyhedral
if it is a finite intersection of halfspaces.} set $D$ and 
$\ve\in\RPP$ such that 
$C\cap \ball{c}{\varepsilon} =D\cap \ball{c}{\varepsilon}$.
\end{definition}

It is clear from the definition that every polyhedron is
polyhedral at each of its points and that 
every subset $C$ of $X$ is polyhedral
at each point in $\inte C$. 

\begin{lemma}
\label{l:lpoly}
Let $C$ be a subset of $X$, and assume that $C$ is polyhedral at
$c \in C$. Then there exist $\ve\in\RPP$, a finite set $I$,
$(d_i)_{i\in I}\in (X\smallsetminus\{0\})^I$,
$(\delta_i)_{i\in I}\in\RR^I$ such that
\begin{equation}
\label{e:0416e1}
C\cap\ball{c}{\ve} = 
\mmenge{x\in X}{\max_{i\in I}
\big(\scal{d_i}{x}-\delta_i\big)\leq 0} \cap \ball{c}{\ve},
\end{equation}
$(\forall i\in I)$ $\scal{d_i}{c}=\delta_i$, 
and 
\begin{equation}
\label{e:0416e2}
\big(\forall y \in C\cap \ball{c}{\varepsilon}\big)\quad 
N_C(y) = \sum_{i\in I(y)}\RP d_i,
\quad \text{where $I(y) := \menge{i\in I}{\scal{d_i}{y}=\delta_i}$.}
\end{equation}
\end{lemma}
\begin{proof}
Combine Lemma~\ref{l:0416a} with \cite[Theorem~6.46]{RW98}. 
\end{proof}

\begin{lemma}
\label{l:proj_poly}
Let $C$ be a nonempty closed convex subset of $X$
that is polyhedral at $c \in C$. 
Then there exists $\varepsilon \in\RPP$ such that
\begin{equation}
\big(\forall x\in P_C^{-1}(C\cap \ball{c}{\ve})\big)\quad
\scal{x -P_Cx}{c -P_Cx} =0 \;\;\text{and}\;\; x -P_Cx \in N_C(c). 
\end{equation}
\end{lemma}
\begin{proof}
We adopt the notation of the conclusion of Lemma~\ref{l:lpoly}. 
Let $x\in X$ such that $y := P_Cx\in\ball{c}{\ve}$.
Then $x-y\in N_C(y)$ and Lemma~\ref{l:lpoly} guarantees the
existence of $(\lambda_i)_{i\in I(x)}\in\RP^{I(y)}$ such that 
\begin{equation}
x -y =\sum_{i\in I(y)} \lambda_id_i.
\end{equation}
Then
\begin{equation}
\scal{x -y}{c -y} 
=\sum_{i\in I(y)} \lambda_i\scal{d_i}{c -y} 
=\sum_{i\in I(y)} \lambda_i(\scal{d_i}{c}-\scal{d_i}{y})
=\sum_{i\in I(y)} \lambda_i(\delta_i-\delta_i) 
=0
\end{equation}
and so $\scal{x-P_Cx}{c-P_Cx} = 0$. 
Furthermore, by Lemma~\ref{l:Ncone}, $x -P_Cx \in N_C(c)$.
\end{proof}

\subsection{Two convex sets}

\begin{proposition}
\label{p:cone}
Let $A$ and $B$ be closed convex subsets of $X$ such
that $A\cap B\neq \varnothing$. 
Then the  following hold:
\begin{enumerate}
\item 
\label{p:cone0}
$(\forall a\in A)(\forall b\in B)\quad 
N_{A-B}(a -b) =N_A(a)\cap \big(-N_B(b)\big)$. 
\item 
\label{p:cone1}
$0 \in \inte(A -B) \quad\Leftrightarrow\quad N_{A-B}(0) =\{0\}$.
\item 
\label{p:cone2}
$\inte B \neq\varnothing \quad\Leftrightarrow\quad (\forall x \in B)\; N_B(x) \cap (-N_B(x)) =\{0\}$.
\end{enumerate}
\end{proposition}
\begin{proof} \
\ref{p:cone0}: 
Noting that $N_{-B}(-b) =-N_B(b)$ by definition,  
the conclusion follows from \cite[Proposition~2.11(ii)]{MN14} or
from a direct computation. 
\ref{p:cone1}: Clear from \cite[Corollary~6.44]{BC11}. 
\ref{p:cone2}: Let $x \in B$. By \ref{p:cone0}, 
$N_B(x) \cap (-N_B(x)) =N_{B-B}(0)$.
Now Lemma~\ref{l:int} and \ref{p:cone1} complete the proof. 
\end{proof}

\begin{corollary}
\label{c:CQ}
Let $A$ be a linear subspace of $X$, and let $B$ be a nonempty
closed convex subset of $X$ such that $A\cap B\neq\varnothing$. 
Then the following hold:
\begin{enumerate}
\item 
\label{c:CQ_ApNB}
$0 \in \inte(A -B) \quad\Leftrightarrow\quad \left[ (\forall x \in A\cap B)\; A^\perp\cap N_B(x) =\{0\} \right]$.
\item 
\label{c:CQ_AintB}
$A\cap \inte B \neq\varnothing \quad\Leftrightarrow\quad \big[ 0
\in \inte(A -B) \quad\text{and}\quad \inte B \neq\varnothing \big]$.
\end{enumerate}
\end{corollary}
\begin{proof} \ 
\ref{c:CQ_ApNB}: Let $x \in A\cap B$. Since $A$ is a linear subspace, 
Proposition~\ref{p:cone}\ref{p:cone0} yields
\begin{equation}
A^\perp\cap N_B(x) =-(N_A(x)\cap (-N_B(x))) =-N_{A-B}(0).
\end{equation}
Now apply Proposition~\ref{p:cone}\ref{p:cone1}.

\ref{c:CQ_AintB}: 
If $0 \in \inte(A -B)$ and $\inte B \neq\varnothing$, 
then $0 \in \reli(A -B)$ and $\reli B=\inte B$.
Since $A$ is a linear subspace, we have 
$\reli A =A$, and using \cite[Corollary~6.6.2]{Roc70}, we get
\begin{equation}
0 \in \reli(A -B) =\reli A -\reli B = A -\inte B,
\end{equation}
which implies $A\cap \inte B \neq\varnothing$. The converse is obvious.
\end{proof}

\begin{lemma}
\label{l:0416a}
Let $A$ and $B$ be closed convex subsets of $X$, and let
$c\in A\cap B$ and $\varepsilon\in\RPP$ be such that
$A\cap\ball{c}{\ve}=B\cap\ball{c}{\ve}$.
Then $N_A(c)=N_B(c)$.
\end{lemma}
\begin{proof}
Let $u\in X$.
Working with the directional derivative and 
using \cite[Proposition~17.17(i)]{BC11}, we have 
$u\in N_A(c) = \partial \iota_A(c)$
$\Leftrightarrow$
$(\forall h\in X)$ $\scal{u}{h}\leq \iota_A'(c;h)$
$\Leftrightarrow$
$(\forall h\in X)$ $\scal{u}{h}\leq \iota_B'(c;h)$
$\Leftrightarrow$
$u\in N_B(c) = \partial \iota_B(c)$.
\end{proof}

\subsection{Monotone operators}

\begin{lemma}
\label{l:block}
Let $L\colon X\times X\to X\times X\colon
(x,y)\mapsto (L_{11}x+L_{12}y,L_{21}x+L_{22}y)$,
where each $L_{ij}\colon X\to X$ is linear.
Assume that 
$L_{11}^*L_{22} +L_{21}^*L_{12} =\Id$ and that
$L_{11}^*L_{21}$ and $L_{22}^*L_{12}$ are skew\footnote{Recall
that $S\colon X\to X$ is skew if $S^*=-S$.}.
Then the following hold:
\begin{enumerate}
\item
\label{l:block1}
If $(x,y)\in X\times X$ and $(u,v)=L(x,y)$,
then 
$\scal{u}{v} =\scal{x}{y}$.
\item
\label{l:block2}
Let $M\colon X\To X$ be a monotone operator, and define 
$M_L\colon X\To X$ via 
$\gra M_L = L(\gra M)$. 
Then for all pairs $(x,y)$ and $(x',y')$ in $\gra M$ 
and $(u,v)=L(x,y)$, $(u',v')=L(x',y')$, we have 
$\scal{u-v}{u'-v'} = \scal{x-x'}{y-y'}$;
consequently, $M_L$ is monotone. 
\end{enumerate}
\end{lemma}
\begin{proof} \
\ref{l:block1}: The assumptions indeed imply 
\begin{subequations}
\begin{align}
\scal{u}{v} &=\scal{L_{11}x +L_{12}y}{L_{21}x +L_{22}y} \\ 
&=\scal{x}{(L_{11}^*L_{22} +L_{21}^*L_{12})y} +\scal{x}{L_{11}^*L_{21}x} +\scal{y}{L_{22}^*L_{12}y} =\scal{x}{y}.
\end{align}
\end{subequations}
\ref{l:block2}: 
Since $(x-x',y-y') = (x,y)-(x',y')$ and $L$ is linear, the result
follows from \ref{l:block1}. 
\end{proof}

\begin{corollary}
\label{c:partial}
Let $A$ be a linear subspace of $X$, 
and let $(x,y)$ and $(x',y')$ be in $X\times X$.
Then the following hold:
\begin{enumerate}
\item $\scal{P_Ay -P_{A^\perp}x}{P_Ax -P_{A^\perp}y} =\scal{y}{x}$ 
\item $\scal{(P_Ay -P_{A^\perp}x)-(P_Ay' -P_{A^\perp}x')}{(P_Ax
-P_{A^\perp}y)-(P_Ax' -P_{A^\perp}y')} = \scal{y-y'}{x-x'}$. 
\end{enumerate}
\end{corollary}
\begin{proof}
Set $L_{11}=L_{22}=P_A$ and $L_{12}=L_{21}=-P_{A^\perp}$ in 
Lemma~\ref{l:block}. 
\end{proof}

\begin{remark}
Spingarn's original \emph{partial inverse} \cite{Spi83}
arises from
Lemma~\ref{l:block} by setting 
$L_{11}=L_{22}=P_A$ and $L_{12}=L_{21}=P_{A^\perp}$
while in in Corollary~\ref{c:partial} we used 
$L_{11}=L_{22}=P_A$ and $L_{12}=L_{21}=-P_{A^\perp}$.
Other choices are possible: e.g., 
if $X=\RR^2$ and $R$ denotes the clockwise rotator by $\pi/4$,
then a valid choice for Lemma~\ref{l:block} is 
$L_{11}=L_{22}=\tfrac{1}{2}R$ and
$L_{12}=L_{21}=\tfrac{1}{2}R^*$.
\end{remark}

\subsection{Finite convergence conditions for the proximal point
algorithm}

\label{s:Luque}

It is known (see, e.g., \cite[Theorem~6]{EB92}) that  
the DRA
is a special case of the exact proximal point algorithm
(with constant parameter 1). The latter generates,
for a given maximally monotone operator $M\colon X\To X$ with resolvent 
$T := J_M = (\Id+M)^{-1}$, a sequence
by 
\begin{equation}
z_0 \in X, \quad (\forall\nnn)\quad z_{n+1} :=Tz_n = (\Id +M)^{-1}z_n,  
\end{equation}
in order to solve the problem
\begin{equation}
\text{find $z\in X$ such that $0\in Mz$; equivalently, $z\in \Fix
T$.}
\end{equation}
A classical sufficient condition dates back to Rockafellar 
(see \cite[Theorem~3]{Roc76}) who proved finite convergence when
\begin{equation}
\label{e:R}
(\exi \bar{z}\in X)(\exi \delta\in\RPP)\quad
\ball{0}{\delta} \subseteq M\bar{z}. 
\end{equation}
It is instructive to view this condition from the resolvent side:
\begin{equation}
\label{e:Rres}
(\exi\bar{z}\in X)(\exi\delta\in\RPP)(\forall z\in X)\quad
\|z-\bar{z}\|\leq\delta \;\Rightarrow\; Tz=\bar{z}.
\end{equation}
Note that since $\Fix T$ is convex, this implies that 
\begin{equation}
\Fix T = \{\bar{z}\}
\end{equation}
is a singleton, 
which severely limits the applicability of this condition. 

Later, Luque (see \cite[Theorem~3.2]{Luq84}) proved finite 
convergence under the more general condition 
\begin{equation}
\label{e:L}
M^{-1}0\neq\varnothing
\;\text{and}\;
(\exi\delta\in\RPP)\;\;
M^{-1}\big(\ball{0}{\delta}\big)\subseteq M^{-1}0.
\end{equation}
On the resolvent side, his condition turns into
\begin{equation}
\label{e:Lres}
\Fix T\neq\varnothing\;\text{and}\;
(\exists \delta >0)\quad \|z -Tz\| \leq\delta \quad\Rightarrow\quad Tz \in \Fix T.
\end{equation}
However, when $M^{-1}0=\Fix T\neq\varnothing$,
it is well known that $z_n-Tz_n\to 0$; thus, 
the finite-convergence condition is essentially a tautology.

When illustrating our main results, we shall provide
examples where both \eqref{e:R} and \eqref{e:L} fail
while our results are applicable (see Remark~\ref{r:Bpoly} and
Example~\ref{ex:epi} below).

\subsection{Douglas--Rachford operator}

For future use, we record some results on the DRA that are easily
checked. 
Recall that, for two nonempty closed convex subsets $A$ and $B$,
the DRA operator is 
\begin{equation}
T :=\Id -P_A +P_B(2P_A -\Id) =\tfrac{1}{2}\left(\Id +R_BR_A\right).
\end{equation}

The following result, the proof of which we omit since
it is a direct verification, records 
properties in the presence of
affinity/linearity.

\begin{proposition}
\label{p:0415a}
Let $A$ be an affine subspace of $X$ and let $B$ be a nonempty closed
convex subset of $X$. 
Then the following hold:
\begin{enumerate}
\item
\label{p:0415a1}
$P_A$ is an affine operator and
\begin{equation}
\label{ep:0415a1}
P_AR_A =P_A, \quad P_AT =P_AP_BR_A \quad\text{and}\quad T =(P_A
+P_B -\Id)R_A.
\end{equation}
\item
\label{p:0415a2}
If $A$ is a linear subspace, then
$P_A$ is a symmetric linear operator and
\begin{equation}
\label{ep:0415a2}
T = \big(P_AP_B -P_{A^\perp}(\Id -P_B)\big)R_A. 
\end{equation}
\end{enumerate}
\end{proposition}

The next result will be used in Section~\ref{s:prodSpin} below to 
clarify the connection between the DRA and Spingarn's method. 

\begin{lemma}
\label{l:SpinDR}
Let $A$ be a linear subspace of $X$, let $B$ be a nonempty
closed convex subset of $X$,
let $(a,a^\perp)\in A\times A^\perp$, and 
set $(a_+,a_+^\perp) := (P_AP_B(a+a^\perp),
P_{A^\perp}(\Id-P_B)(a+a^\perp))\in A\times A^\perp$.
Then $T(a-a^\perp) = a_+-a_+^\perp$. 
\end{lemma}
\begin{proof}
Clearly, $a_+\in A$ and $a_+^\perp\in A^\perp$. 
Since 
$R_A(a-a^\perp) = (P_A-P_{A^\perp})(a-a^\perp) = a+ a^\perp$,
the conclusion follows from \eqref{ep:0415a2}.
\end{proof}

\section{The affine-polyhedral case with Slater's condition}

\label{s:main1}

In this section, we are able to state and prove
finite convergence of the DRA in the case where
$A$ is an affine subspace and $B$ is a polyhedral set
such that Slater's condition, $A\cap\inte B\neq\varnothing$,
is satisfied. We start by recalling our standing assumptions.
We assume that
\begin{empheq}[box=\mybluebox]{equation}
\text{$X$ is a finite-dimensional real Hilbert space,}
\end{empheq}
and that
\begin{empheq}[box=\mybluebox]{equation}
\text{$A$ and $B$ are closed convex subsets of $X$ such that $A\cap B \neq\varnothing$}.
\end{empheq}
The DRA is based on the operator
\begin{empheq}[box=\mybluebox]{equation}
T :=\Id -P_A +P_BR_A.
\end{empheq}
Given a starting point $z_0 \in X$, 
the DRA sequence $(z_n)_\nnn$ is generated by
\begin{subequations}
\label{e:DRAseqs}
\begin{empheq}[box=\mybluebox]{equation}
\label{e:DRAseqsT}
(\forall\nnn)\quad z_{n+1} :=Tz_n.
\end{empheq}
We also set 
\begin{empheq}[box=\mybluebox]{equation}
(\forall\nnn)\quad a_n :=P_Az_n, \quad r_n :=R_Az_n =2a_n -z_n.
\end{empheq}
\end{subequations}

We now state the basic convergence result for the DRA. 

\begin{fact}[convergence of the DRA]
\label{f:cvgFixT}
The DRA sequences \eqref{e:DRAseqs} satisfy
\begin{equation}
z_n \to z \in \Fix T =(A\cap B) +N_{A-B}(0) \quad\text{and}\quad a_n 
\to P_Az \in A\cap B.  \end{equation}
\end{fact}
\begin{proof}
Combine \cite[Corollary~3.9 and Theorem~3.13]{BCL04}.
\end{proof}

Fact~\ref{f:cvgFixT} can be strengthened
when a constraint qualification is satisfied.

\begin{lemma}
\label{l:cvg}
Suppose that $0 \in \inte(A -B)$.
Then there exists a point $c\in A\cap B$ such that 
the following hold for the DRA sequences \eqref{e:DRAseqs}:
\begin{enumerate}
\item 
\label{l:cvg_gen}
$\reli(A)\cap\reli(B)\neq\varnothing$ and hence 
$(z_n)_\nnn$, $(a_n)_\nnn$ and $(r_n)_\nnn$ converge linearly to
$c$. 
\item 
\label{l:cvg_int}
If $c \in \inte B$, then the convergence of $(z_n)_\nnn$, $(a_n)_\nnn$ and $(r_n)_\nnn$ to $c$ is finite.
\end{enumerate}
\end{lemma}
\begin{proof}
\ref{l:cvg_gen}: From $0 \in \inte(A -B)$, we have $N_{A-B}(0) =\{0\}$ 
due to Proposition~\ref{p:cone}. 
This gives $\Fix T =A\cap B$, and Fact~\ref{f:cvgFixT} implies
$z_n \to c \in A\cap B$ and $a_n \to P_Ac =c$. Then $r_n =2a_n -z_n \to c$.
Using again $0 \in \inte(A-B)$, 
\cite[Corollary~6.6.2]{Roc70} yields $0 \in \reli(A -B) =\reli A -\reli B$, 
and thus $\reli A\cap \reli B \neq\varnothing$.
Now the linear convergence follows from \cite[Theorem~4.14]{Pha14} 
or from \cite[Theorem~8.5(i)]{BNP15}. 

\ref{l:cvg_int}: 
Since $r_n \to c$ and $a_n \to c$ by \ref{l:cvg_gen}, 
there exists $n \in \NN$ such that
$r_n \in B$ and $a_n \in B$.
Then $P_Br_n =r_n$ and 
\begin{equation}
z_{n+1} =z_n -a_n +P_Br_n =z_n -a_n +r_n =a_n \in A\cap B = \Fix
T.
\end{equation} 
Hence $a_n=z_{n+1}=z_{n+2}=\cdots$ and we are done. 
\end{proof}

\begin{lemma}
Suppose that $A$ is a linear subspace.
Then the DRA sequences \eqref{e:DRAseqs} satisfy
\begin{equation}
\label{e:anan+}
a_n =P_Az_n =P_Ar_n \quad\text{and}\quad 
a_{n+1} =P_ATz_n = P_AP_BR_Az_n = P_AP_Br_n,
\end{equation}
and 
\begin{equation}
\label{e:aa}
(\forall\nnn)\quad a_n -a_{n+1} =P_A(r_n -P_Br_n). 
\end{equation}
\end{lemma}
\begin{proof}
\eqref{e:anan+}: Clear from \eqref{ep:0415a1}.
\eqref{e:aa}: Use \eqref{e:anan+} and the linearity of $P_A$.
\end{proof}

\begin{lemma}
\label{l:afinite}
Suppose that $A$ is a linear subspace and that 
for the DRA sequences \eqref{e:DRAseqs} there exists $p \in \NN$ 
such that $a_p =a_{p+1} =c \in A\cap B$, 
and that there is a subset $N$ of $X$ such that $r_p -P_Br_p \in N$ 
and $A^\perp \cap N =\{0\}$. 
Then $(\forall n\geq p+1)$ $z_{n}=c$. 
\end{lemma}
\begin{proof} 
Since $a_p -a_{p+1} =0$, \eqref{e:aa} 
implies $r_p -P_Br_p \in A^\perp \cap N=\{0\}$. 
Thus $P_Br_p =r_p$ and therefore
$z_{p+1} =z_p - a_p + r_p =a_p =c \in A\cap B \subseteq \Fix T$.
\end{proof}

\begin{lemma}
\label{l:prod}
Suppose that $A$ is a linear subspace and 
let $a \in A$. Then the DRA sequence \eqref{e:DRAseqsT} satisfies
\begin{subequations}
\begin{equation}
(\forall\nnn)\quad \scal{z_n -z_{n+1}}{z_{n+1} -a} =\scal{r_n -P_Br_n}{P_Br_n -a}, 
\end{equation}
and
\begin{equation}
(\forall\nnn)(\forall m\in\NN)
\quad \scal{z_{n+1}-z_{m+1}}{(z_n-z_{n+1})-(z_m-z_{m+1})} \geq 0.
\end{equation}
\end{subequations}
\end{lemma}
\begin{proof}
Let $n\in\NN$. 
Using \eqref{e:anan+}, we find that 
\begin{equation}
z_n -z_{n+1} =a_n -P_Br_n =P_Ar_n -P_Br_n 
=P_A(r_n -P_Br_n) -P_{A^\perp}(P_Br_n -a).
\end{equation}
Next, \eqref{ep:0415a2} yields
\begin{equation}
z_{n+1} -a =P_AP_Br_n -P_{A^\perp}(r_n -P_Br_n) -a 
=P_A(P_Br_n -a) -P_{A^\perp}(r_n -P_Br_n).
\end{equation}
Moreover, $r_n -P_Br_n \in N_B(P_Br_n) =N_{B-a}(P_Br_n -a)$
and $N_{B-a}$ is a monotone operator. 
The result thus follows from 
Corollary~\ref{c:partial}. 
\end{proof}

\begin{lemma}
\label{l:lfun}
Suppose that $A$ is a linear subspace and that 
the DRA sequences \eqref{e:DRAseqs} satisfy
\begin{equation}
\label{e:scal0}
z_n \to a \in A\quad\text{and}\quad
(\forall\nnn)\quad \scal{r_n -P_Br_n}{a -P_Br_n} =0.
\end{equation} 
Then there is no linear functional $f \colon X \to \RR$ such that
\begin{equation}
\label{e:fdecrease}
(\forall\nnn)\quad f(z_n) >f(z_{n+1}).
\end{equation}
\end{lemma}
\begin{proof}
Suppose to the contrary that there exists a linear function $f
\colon X \to \RR$ satisfying \eqref{e:fdecrease}.
Now set 
\begin{equation} 
(\forall n \in \NN\smallsetminus \{0\})\quad y_n
:=z_{n-1} -z_n \quad\text{and}\quad w_n :=z_n -a.  
\end{equation}
On the one hand, Lemma~\ref{l:plfun} yields
$n_0 \in \NN\smallsetminus \{0\}$ and 
$(\mu_1, \dots, \mu_{n_0}) \in \RP^{n_0}$ such that 
\begin{equation}
\label{e:SyiSwi}
\sum_{k=1}^{n_0}\mu_k=1\quad\text{and}\quad
\scal{\sum_{k=1}^{n_0} \mu_ky_k}{\sum_{k=1}^{n_0} \mu_kw_k} >0.
\end{equation}
On the other hand, Lemma~\ref{l:prod} and \eqref{e:scal0} yield
\begin{equation}
(\forall k \in \NN\smallsetminus \{0\})(\forall j \in \NN\smallsetminus \{0\})\quad \scal{y_k}{w_k} =0 \quad\text{and}\quad \scal{y_k -y_j}{w_k -w_j} \geq 0;
\end{equation}
consequently, with the help of \cite[Lemma~2.13(i)]{BC11}, 
\begin{equation}
\label{e:mom-1}
\scal{\sum_{k=1}^{n_0} \mu_k y_k}{\sum_{k=1}^{n_0} \mu_k w_k}
=\sum_{k=1}^{n_0} \mu_k\scal{y_k}{w_k}
-\tfrac{1}{2}\sum_{k=1}^{n_0}\sum_{j=1}^{n_0} \mu_k\mu_j\scal{y_k -y_j}{w_k -w_j} \leq 0.
\end{equation}
Comparing \eqref{e:SyiSwi} with \eqref{e:mom-1}, we arrive at the
desired contradiction. 
\end{proof}

We are now ready for our first main result concerning the finite
convergence of the DRA. 

\begin{theorem}[finite convergence of DRA in the
affine-polyhedral case]
\label{t:Bpoly}
Suppose that $A$ is a affine subspace,
that $B$ is polyhedral at every point in $A\cap \bd B$,
and  that \emph{Slater's condition}
\begin{equation}
\label{e:AintB}
A\cap \inte B \neq\varnothing
\end{equation}
holds.
Then the DRA sequences \eqref{e:DRAseqs}
converge in \emph{finitely many steps} to a point in $A\cap B$.
\end{theorem}
\begin{proof}
After translating the sets if necessary, we can and do assume
that $A$ is a linear subspace of $X$.
By Corollary~\ref{c:CQ}\ref{c:CQ_AintB}, \eqref{e:AintB} yields
\begin{equation}
0 \in \inte(A -B) \quad\text{and}\quad \inte B \neq\varnothing.
\end{equation}
Lemma~\ref{l:cvg}\ref{l:cvg_gen} thus implies that 
$(z_n)_\nnn$, $(a_n)_\nnn$ and $(r_n)_\nnn$ converge linearly to a point 
$c \in A\cap B$.
Since $P_B$ is (firmly) nonexpansive, it also follows that
$(P_Br_n)_\nnn$ converges linearly to $c$.
Since $B$ is clearly polyhedral at every point in $\inte B $ it
follows from the hypothesis that
$B$ is polyhedral at $c$.
Lemma~\ref{l:proj_poly} guarantees the existence of 
$n_0 \in \NN$ such that 
\begin{equation}
\label{e:scal=0}
(\forall n \geq n_0)\quad \scal{r_n -P_Br_n}{c -P_Br_n} =0,
\end{equation}
and
\begin{equation}
\label{e:xnPBxna}
(\forall n \geq n_0)\quad r_n -P_Br_n \in N_B(c).
\end{equation}
Because of $\inte B \neq\varnothing$ and $c\in B$, 
Proposition~\ref{p:cone}\ref{p:cone2} yields 
$N_B(c)\cap (-N_B(c)) =\{0\}$.
Hence $N_B(c)$ is a nonempty pointed closed convex cone.
Using again $0 \in \inte(A -B)$, 
Corollary~\ref{c:CQ}\ref{c:CQ_ApNB} gives
\begin{equation}
\label{e:0417a}
\ker P_A \cap N_B(c) =A^\perp \cap N_B(c) =\{0\}.
\end{equation}
In view of Lemma~\ref{l:pointed}\ref{l:pointed2},
\begin{equation}
\label{e:0417b}
K :=P_A\big(N_B(c)\big) \text{ is a nonempty pointed closed
convex cone.}
\end{equation}
Combining \eqref{e:aa} and \eqref{e:xnPBxna}, we obtain
\begin{equation}
\label{e:aaK'}
(\forall n \geq n_0)\quad a_n -a_{n+1} \in K.
\end{equation}
Since $a_n \to c$ and $K$ is a closed convex cone, we have 
\begin{equation}
\label{e:asK}
(\forall n \geq n_0)\quad a_n -c =\sum_{k=n}^\infty (a_k  -a_{k+1}) \in K.
\end{equation} 
We now consider two cases.

\emph{Case~1:} $(\exi p\geq n_0)$ $a_p=c$. \\
Using \eqref{e:0417b} and 
\eqref{e:aaK'}, we deduce from Lemma~\ref{l:stop} that $a_p =a_{p+1} =c$.
Now \eqref{e:xnPBxna}, \eqref{e:0417a}, and Lemma~\ref{l:afinite}
yield $(\forall n\geq p+1)$ $z_n=c$ as required.

\emph{Case~2:} $(\forall n\geq n_0)$ $a_n\neq c$. \\
By \eqref{e:asK}, $(\forall n\geq n_0)$ $a_n -c \in K\smallsetminus\{0\}$.
Since $K$ is pointed (see \eqref{e:0417b}),
Lemma~\ref{l:cvxcone} yields
$v \in \reli K\cap \reli K^\oplus\subseteq K$ such that
\begin{equation}
\label{e:0417c}
(\forall n \geq n_0)\quad \scal{v}{a_n -c} >0.
\end{equation}
Recalling \eqref{e:0417b}, we get $u\in N_B(c)$ such that 
$v=P_Au$. 
Clearly, 
$(\forall n \geq n_0)$ $\scal{u}{a_n -c} 
= \scal{u}{P_A(a_n -c)} = \scal{P_Au}{a_n -c} = \scal{v}{a_n -c}$.
It follows from \eqref{e:0417c} that 
\begin{equation}
\label{e:scal>0}
(\forall n \geq n_0)\quad \scal{u}{a_n -c} > 0. 
\end{equation}
Since $u \in N_B(c)$ we also have 
\begin{equation}
\label{e:scal'}
(\forall n \geq n_0)\quad \scal{u}{P_Br_n -c} \leq 0.
\end{equation}
Now define a linear functional on $X$ by 
\begin{equation}
f \colon X \to \RR \colon x \mapsto \scal{u}{x}.
\end{equation}
In view of \eqref{e:scal>0} with \eqref{e:scal'}, we obtain
$(\forall n\geq n_0)$
$f(z_n -z_{n+1}) =f(a_n -P_Br_n) =f(a_n -c) -f(P_Br_n -c) 
=\scal{u}{a_n -c} -\scal{u}{P_Br_n -c} >0$.
Therefore, 
\begin{equation}
(\forall n \geq n_0)\quad f(z_n) >f(z_{n+1}).
\end{equation}
However, this and \eqref{e:scal0} together contradict
Lemma~\ref{l:lfun} (applied to $(z_n)_{n \geq n_0}$).
We deduce that \emph{Case~2} never occurs which completes the
proof of the theorem.
\end{proof}

\begin{remark}
\label{r:Bpoly}
Some comments on Theorem~\ref{t:Bpoly} are in order.
\begin{enumerate}
\item
Suppose we replace the Slater condition ``$A\cap\inte
B\neq\varnothing$'' by ``$A\cap \reli B \neq \varnothing$''.
Then one still obtains \emph{linear convergence}
(see \cite[Theorem~4.14]{Pha14} or \cite[Theorem~8.5(i)]{BNP15});
however, 
\emph{finite convergence fails} in general: indeed,
$B$ can chosen to be an affine subspace 
(see \cite[Section~5]{BBNPW14}). 
\item
Under the assumptions of Theorem~\ref{t:Bpoly},
Rockafellar's condition \eqref{e:R} is only applicable when
both 
$A=\{\bar{a}\}$ and $\bar{a}\in \inte B$ hold.
There are many examples where this condition is violated yet
our condition is applicable (see, e.g., the scenario in the
following item). 
\item 
Suppose that $X =\RR^2$, that $A =\RR\times \{0\}$ and 
that $B =\epi f$, where $f\colon X\to \RR\colon x\mapsto |x|-1$.
It is clear that $B$ is polyhedral and that $A\cap\inte
B\neq\varnothing$. 
Define $(\forall\ve\in\RPP)$
$z_\ve := (1 +\varepsilon, \varepsilon)$. 
Then
\begin{equation}
(\forall\ve\in\RPP)\quad
Tz_{\ve} = (1+\ve,\ve)\notin\Fix T = [-1,1]\times\{0\}
\;\;\text{yet}\;\;
\|z_\ve-Tz_{\ve}\|=\ve.
\end{equation}
We conclude that (the resolvent reformulation of) Luque's condition 
\eqref{e:Lres} fails. 
\end{enumerate}
\end{remark}

\section{Applications}

\label{s:four}

\subsection{Product space setup and Spingarn's method}

\label{s:prodSpin}

Let us now consider a feasibility problem with possibly more than
two sets, say 
\begin{equation}
\label{e:moresets}
\text{find a point in } C :=\bigcap_{j=1}^M C_j,
\end{equation}
where
\begin{empheq}[box=\mybluebox]{equation}
\label{e:prodsetup}
\text{$C_1, \dots, C_M$ are closed convex subsets of $X$ such that $C \neq\varnothing$}.
\end{empheq}
This problem is reduced to a two-set problem as follows.
In the product Hilbert space $\bX :=X^M$,
with the inner product defined by
$((x_1, \dots, x_M), (y_1, \dots, y_M)) \mapsto
\sum_{j=1}^M \scal{x_j}{y_j}$,
we set
\begin{equation}
\bA :=\menge{(x,\dots, x) \in \bX}{x \in X} \quad\text{and}\quad 
\bB :=C_1\times \cdots \times C_M.
\end{equation} 
Because of 
\begin{equation}
x \in \bigcap_{j=1}^M C_j \quad\Leftrightarrow\quad (x, \dots, x) \in \bA\cap \bB,
\end{equation}
the $M$-set problem \eqref{e:moresets}, which is formulated in $X$, is
equivalent to the two-set problem 
\begin{equation}
\label{e:prob'}
\text{find a point in $\bA\cap \bB$},
\end{equation}
which is posed in $\bX$.
By, e.g., \cite[Proposition~25.4(iii) and Proposition~28.3]{BC11}, 
the projections of $\bx =(x_1, \dots, x_M) \in \bX$ onto $\bA$ and $\bB$ are respectively given by
\begin{equation}
\label{e:projAB}
P_\bA\bx =\Big( \frac{1}{M}\sum_{j=1}^M x_j, \dots,
\frac{1}{M}\sum_{j=1}^M x_j \Big)
\quad\text{and}\quad P_\bB\bx =\big( P_{C_1}x_1, \dots,
P_{C_M}x_M \big).
\end{equation}
This opens the door of applying the DRA in $\bX$:
indeed, set 
\begin{equation}
\bT :=\Id -P_\bA +P_\bB R_\bA,
\end{equation}
fix a starting point $\bz_0 \in \bX$, and generate
the DRA sequence $(\bz_n)_\nnn$ via
\begin{equation}
\label{e:DRA'}
(\forall\nnn)\quad \bz_{n+1} := \bT\bz_n.
\end{equation}
We now obtain the following result as a consequence of
Theorem~\ref{t:Bpoly}.

\begin{corollary}
Suppose that $C_1,\ldots,C_M$ are polyhedral such that
$\inte C\neq\varnothing$. 
Then the DRA sequence defined by \eqref{e:DRA'} converges
finitely to $\bz =(z, \dots, z) \in \bA\cap \bB$ with 
$z \in C =\bigcap_{j=1}^M C_j$.
\end{corollary}
\begin{proof}
Since $\inte C \neq\varnothing$, 
there exists $c \in C$ and $\ve\in \RPP$ such that 
$(\forall j\in\{1,\ldots,M\})$
$\ball{c}{\ve} \subseteq C$. 
Then $(c, \dots, c) \in \bA\cap \inte\bB$, and so 
$\bA\cap \inte\bB \neq\varnothing$.
Since $\bB =C_1\times \cdots \times C_M$ is a polyhedral subset
of $\bX$, 
the conclusion now follows from Theorem~\ref{t:Bpoly}.
\end{proof}

Problem \eqref{e:moresets} was already considered by Spingarn \cite{Spi85}
for the case where all sets $C_1, \dots, C_M$ are \emph{halfspaces}. 
He cast the resulting problem into the form
\begin{equation}
\text{find $(\ba, \bb) \in \bA\times \bA^\perp$ such that $\bb
\in N_\bB(\ba)$,}
\end{equation}
and suggested solving it by a version of his 
\emph{method of partial inverses} \cite{Spi83},
which generates a sequence $(\ba_n,\bb_n)_\nnn$ via 
\begin{equation}
\label{e:MPI}
(\ba_0, \bb_0) \in \bA\times \bA^\perp 
\quad\text{and}\quad (\forall\nnn) \quad
\begin{cases}
\ba'_n :=P_\bB(\ba_n +\bb_n), & \bb'_n :=\ba_n +\bb_n -\ba'_n, \\
\ba_{n+1} :=P_\bA\ba'_n, & \bb_{n+1} :=\bb'_n -P_\bA\bb'_n.
\end{cases} 
\end{equation}
It was pointed out in \cite[Section~1]{LS87}, \cite[Section~5]{EB92}, \cite[Appendix]{MOT95} and \cite[Remark~2.4]{Com09} 
that this scheme is closely related to the DRA.
(In fact, Lemma~\ref{l:SpinDR} above makes it completely clear why 
\eqref{e:MPI} is equivalent to applying the DRA to $\bA$ and
$\bB$, with starting point $(\ba_0-\bb_0)$.)
However, Spingarn's proof of finite convergence in \cite{Spi85} requires 
\begin{equation}
\ba_n -\ba_{n+1} \in N_C(c) \times \cdots \times N_C(c),
\end{equation}
and he chooses a linear functional $f$ based on the 
``diagonal'' structure of $\bA$ ---  
unfortunately, his proof does not work 
for problem \eqref{e:prob}. Our proof in the previous section
at the same time simplifies and strengthens his proof technique
to allow us to deal with polyhedral sets rather than just
halfspaces. While every polyhedron is an intersection of
halfspaces, the problems are theoretically
equivalent --- in practice, however, there can be \emph{huge
savings} as the requirement to work in Spingarn's setup might
lead to much larger instances of the product space $\bX$! 
It also liberates us from being forced to work in the
product space. Our extension is also intrinsically more 
flexible as the following example illustrates. 

\begin{example}
Suppose that $X =\RR^2$, 
that $A =\menge{(x, x)}{x \in \RR}$ is diagonal, 
and that $B =\menge{(x, y)\in\RR^2}{-y \leq x \leq 2}$.
Clearly, $B$ is polyhedral and $A\cap\inte B\neq\varnothing$. 
Moreover, $B$ is also \emph{not} the Cartesian product of 
two polyhedral subsets of $\RR$, i.e., of two intervals. 
Therefore, the proof of finite convergence of the DRA 
in \cite{Spi85} no longer applies.
However, 
$A$ and $B$ satisfy all assumptions of Theorem~\ref{t:Bpoly},
and thus the DRA finds a point in $A\cap B$ after a finite number
of steps (regardless of the location of starting point).
See Figure~\ref{fg:ex} for an illustration, created with
\texttt{GeoGebra} \cite{GGB}. 
\begin{figure}[!ht]
\centering
\includegraphics[width =0.95\columnwidth]{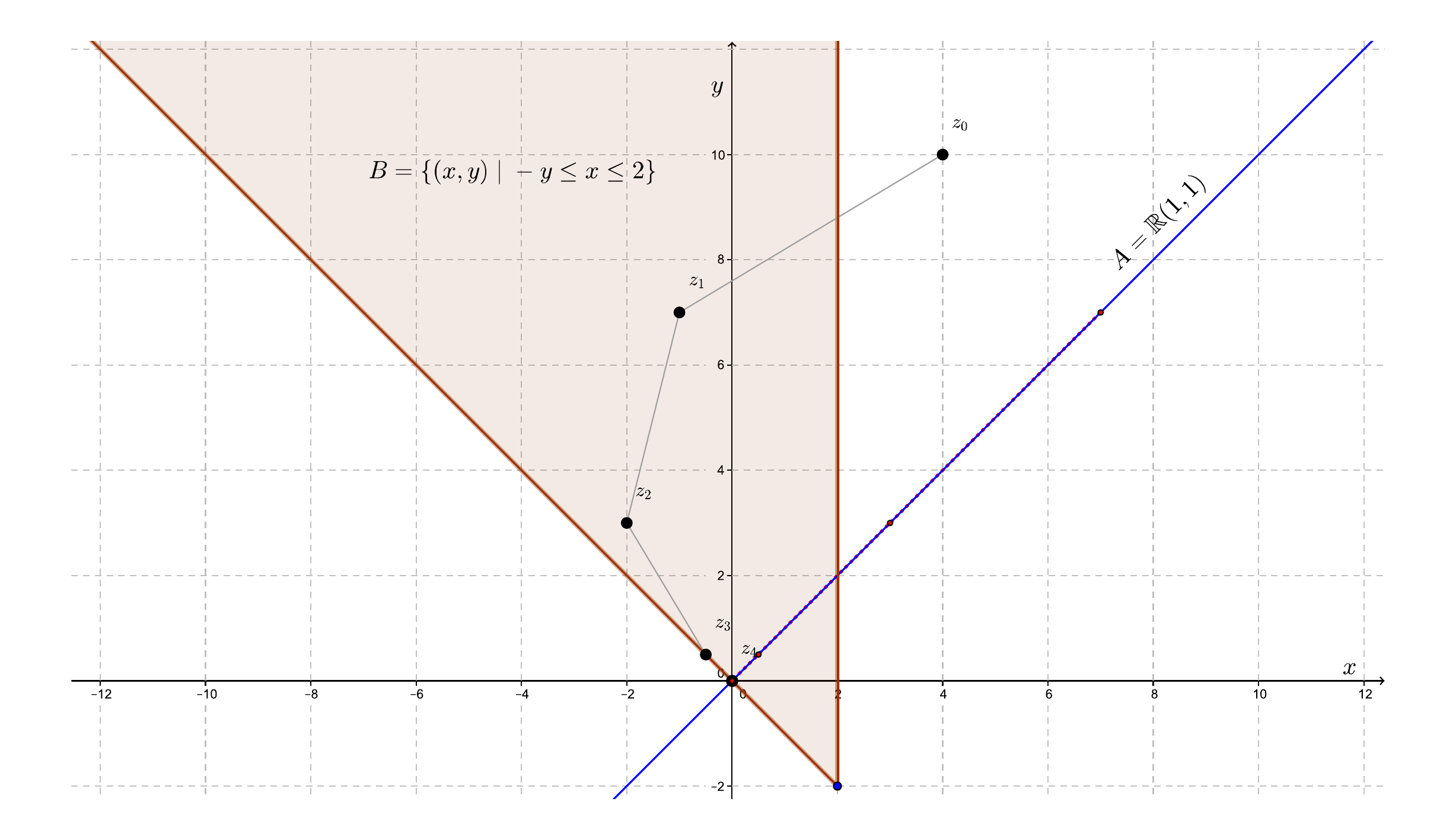}
\caption{The DRA for the case when $B$ is not a Cartesian
product.}
\label{fg:ex}
\end{figure}
\end{example}

\subsection{Solving linear equalities with a strict positivity constraint} 

In this subsection, we assume that
\begin{equation}
X = \RR^N,
\end{equation}
and that 
\begin{equation}
A = \menge{x\in X}{Lx=a}
\;\;\text{and}\;\; 
B = \RR^N_+,
\end{equation}
where $L\in\RR^{M\times N}$ and $a\in\RR^M$.
Note that the set $A\cap B$ is polyhedral yet 
$A\cap B$ has empty interior (unless $A=X$ which is
a case of little interest). 
Thus, Spingarn's finite-convergence result is \emph{never
applicable}.
However, Theorem~\ref{t:Bpoly} guarantees finite convergence of
the DRA provided that 
Slater' condition
\begin{equation}
\label{e:CQ_orthant}
\RR^N_{++} \cap L^{-1}a \neq\varnothing
\end{equation}
holds. 
Using \cite[Lemma~4.1]{BK04}, we obtain 
\begin{equation}
P_A\colon X\to X\colon x\mapsto x -L^\dagger(Lx -a),
\end{equation}
where $L^\dagger$ denotes the Moore-Penrose inverse of $L$.
We also have (see, e.g., \cite[Example~6.28]{BC11})
\begin{equation}
P_B\colon X\to X\colon 
x =(\xi_1, \dots, \xi_N) \mapsto 
x =(\xi_1^+, \dots, \xi_N^+),
\end{equation}
where $\xi^+ = \max\{\xi,0\}$ for every $\xi\in\RR$. 
This implies 
$R_A\colon X\to X\colon
x\mapsto x-2L^\dagger(Lx-a)$ and 
$R_B\colon X\to X\colon 
x =(\xi_1, \dots, \xi_N) \mapsto (|\xi_1|,\ldots,|\xi_N|)$.
We will compare three algorithms, all of which generate a
governing sequence with starting point $z_0\in X$ via
\begin{equation}
(\forall\nnn)\quad z_{n+1}=Tz_n.
\end{equation}
The DRA uses, of course,
\begin{equation}
\label{e:0421DRA}
T = \Id-P_A + P_BR_A.
\end{equation}
The second method is the classical \emph{method of alternating
projections (MAP)} where 
\begin{equation}
T = P_AP_B;
\end{equation}
while the third \emph{method of reflection-projection (MRP)} employs
\begin{equation}
T = P_AR_B. 
\end{equation}

We now illustrate the performance of these three algorithms
numerically. For the remainder of this section, we assume that 
that $X =\RR^2$, 
that $A =\menge{(x, y) \in \RR^2}{x +5y =6}$, and that $B =\RP^2$.
Then $A$ and $B$ satisfy \eqref{e:CQ_orthant},
and the sequence \eqref{e:0421DRA} generated by the 
DRA thus converges finitely to a point in $A\cap B$ regardless of 
the starting point.
See Figure~\ref{fg:orthant} for an illustration, created with \texttt{GeoGebra} \cite{GGB}.
\begin{figure}[!ht]
\centering
\includegraphics[width =0.95\columnwidth, trim=0.5cm 0.5cm 0.5cm 0.5cm]{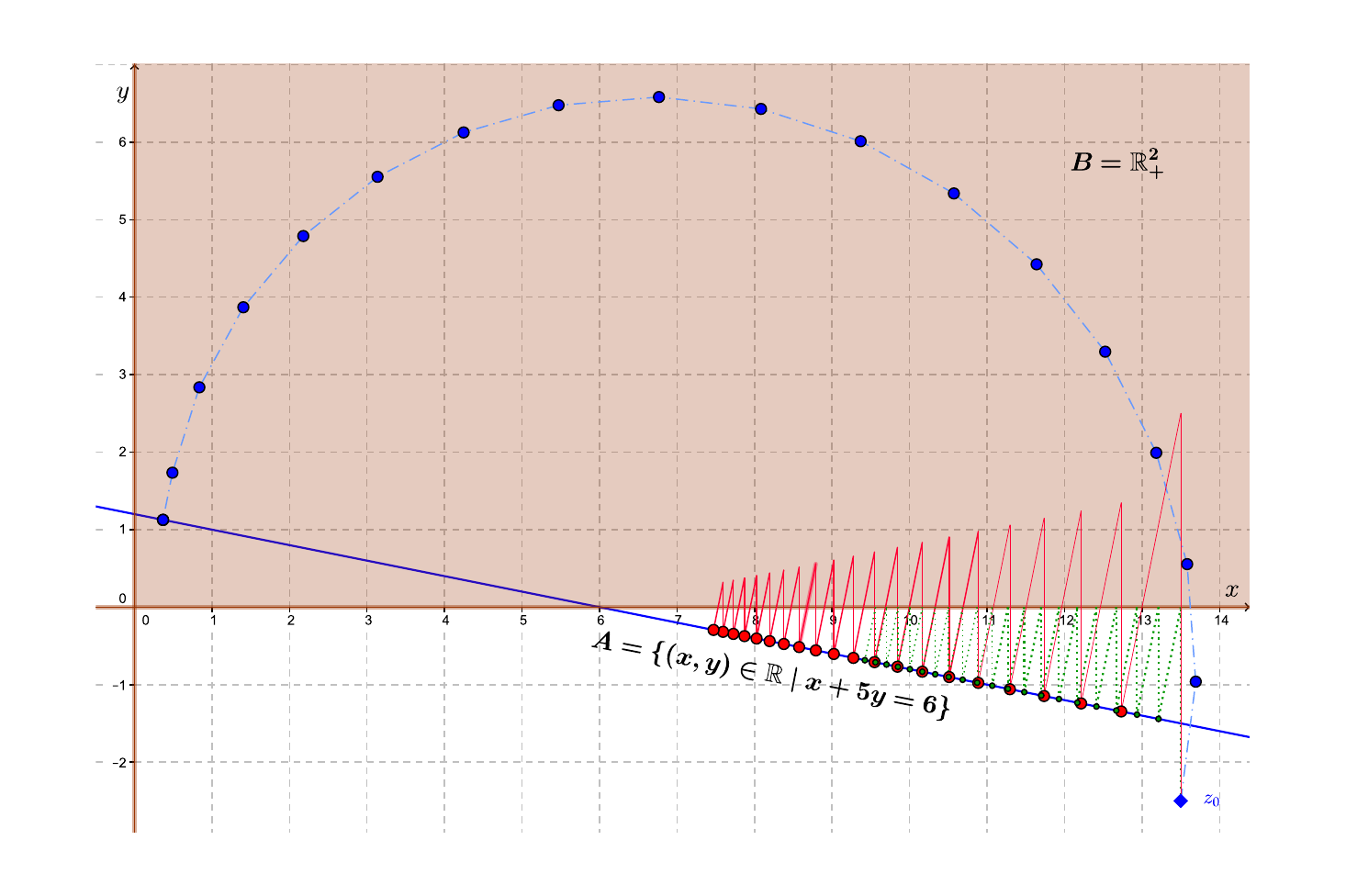}
\caption{The orbits of DRA (blue), MAP (green) and MRP (red).}
\label{fg:orthant}
\end{figure}
Note that the shadow sequence $(P_Az_n)_\nnn$ for the DRA finds a
point in $A\cap B$ even before a fixed point is reached. 
For each starting point $z_0 \in [-100, 100]^2 \subseteq \RR^2$, 
we perform the DRA until $z_{n+1} =z_n$, 
and run the MAP and the MRP until $d_B(z_n) =\max\{d_A(z_n),
d_B(z_n)\} <\ve$, where we set the tolerance $\ve =10^{-4}$. 
Figure~\ref{fg:orthant_steps} compares the number of iterations
needed to stop each algorithm. Note that
even though we put the DRA at an ``unfair disadvantage'' (it must find a
true fixed point while the MAP and the MRP will stop with $\ve$-feasible
solutions), it does extremely well. 
In Figure~\ref{fg:orthant_d}, we level the playing field and  compare the 
distance from $P_Az_n$ (for the DRA) or from $z_n$ (for the MAP
and the MRP) to $B$,
where $n \in \{5, 10\}$.
\begin{figure}[!ht]
\centering
\includegraphics[width =0.95\columnwidth]{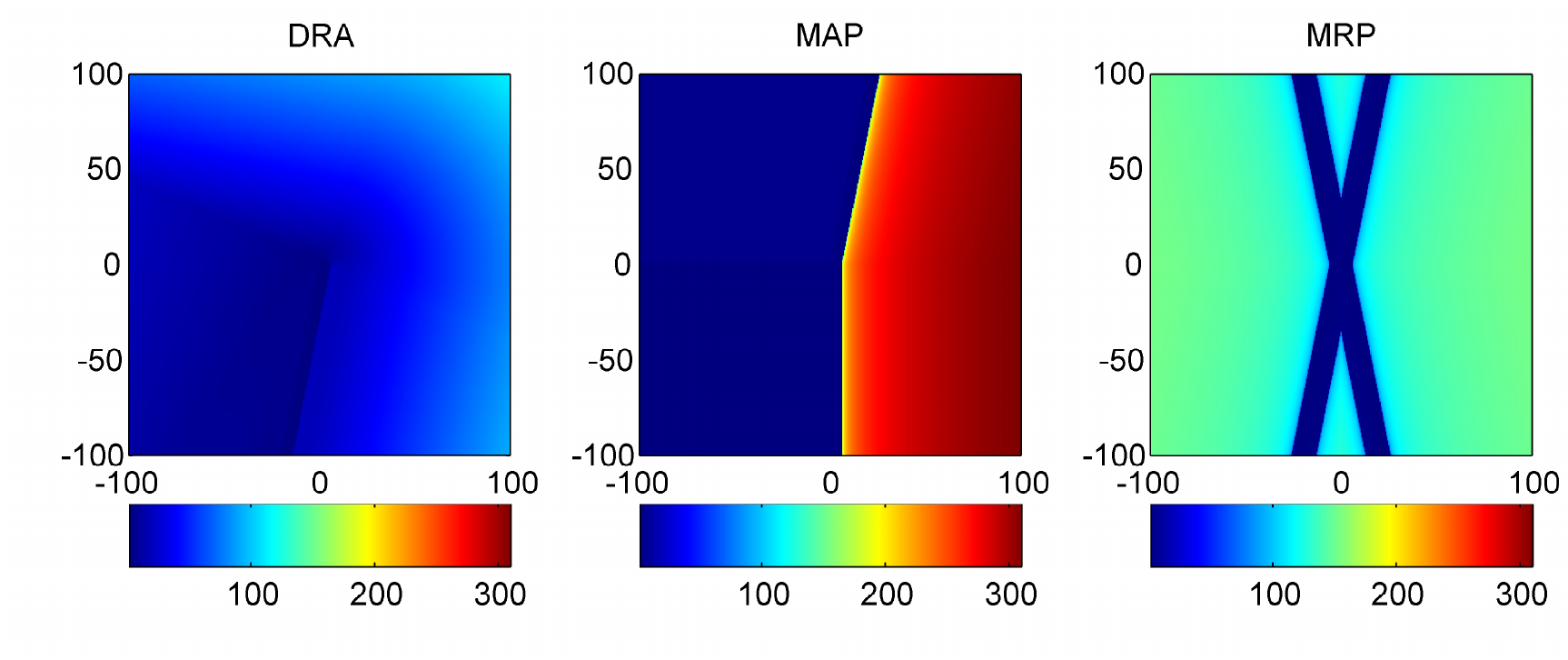}
\caption{Number of iterations needed to get the solution of DRA, MAP and MRP.}
\label{fg:orthant_steps}
\end{figure}

\begin{figure}[!ht]
\centering
\includegraphics[width =0.95\columnwidth]{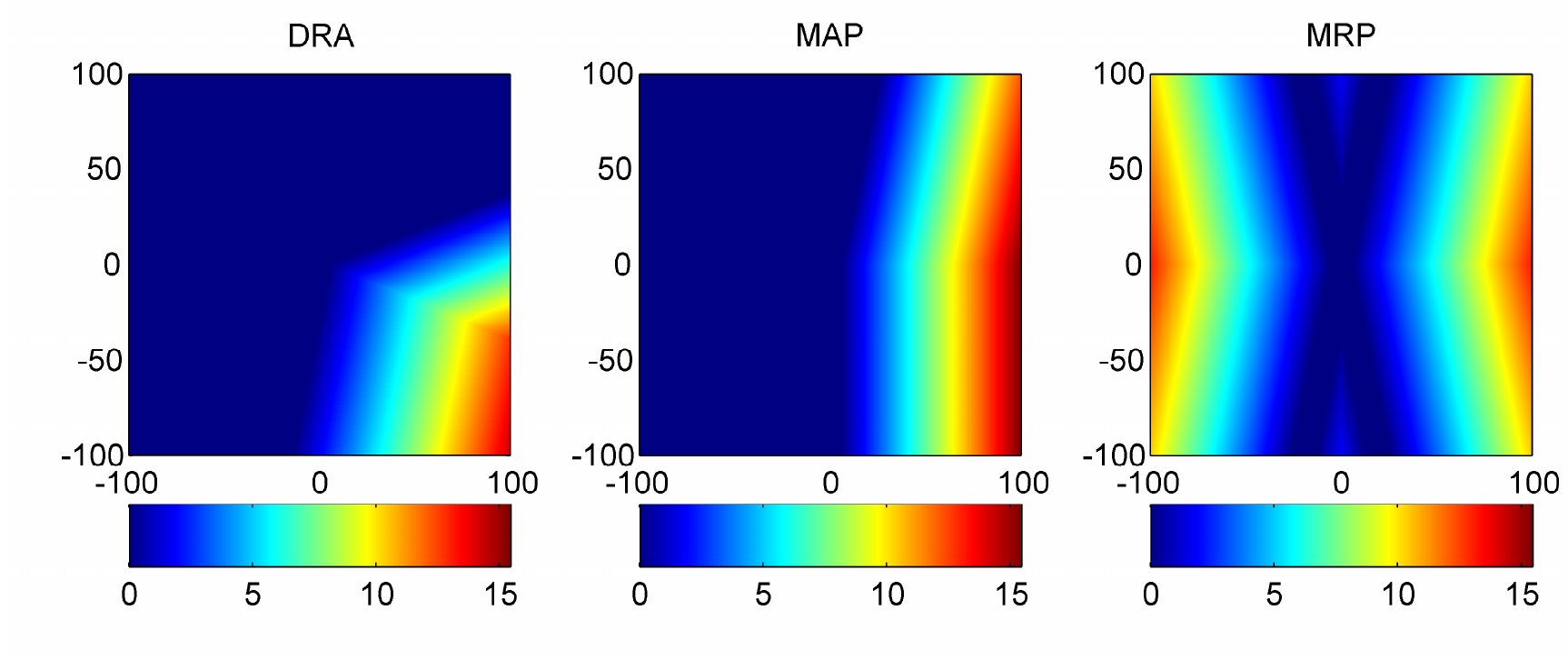} \\
\includegraphics[width =0.95\columnwidth]{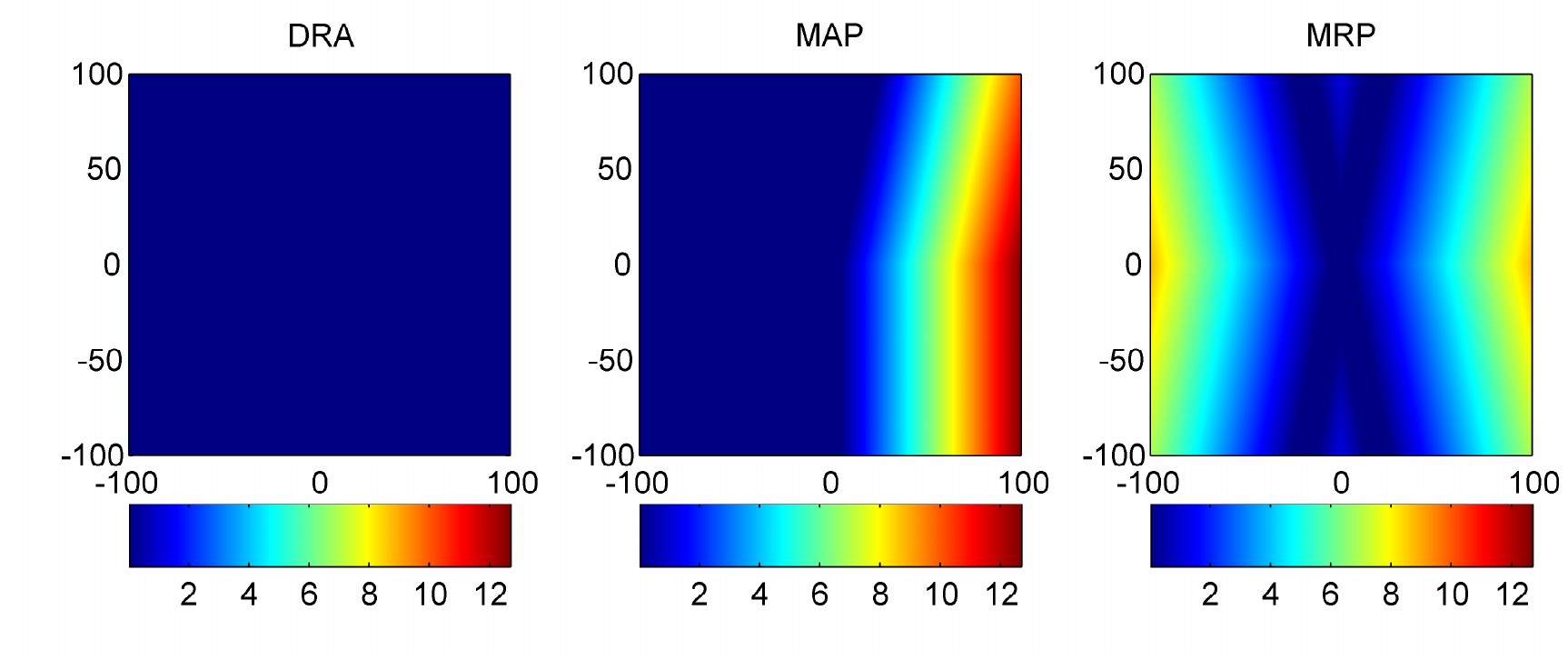}
\caption{The distance from the monitored iterate to $B$ after $5$ steps (\emph{top}) and $10$ steps (\emph{bottom}).}
\label{fg:orthant_d}
\end{figure}

Now we look at the process of reaching a solution for each algorithm. 
For the DRA, we monitor the shadow sequence $(P_Az_n)_\nnn$ and
for the MAP and the MRP, we monitor
$(z_n)_\nnn$. 
Note all three monitored sequences lie in $A$, and we thus are
concerned about the distance to $B$.
Our stopping criterion is that
\begin{equation}
d_B(P_Az_n) < \ve
\end{equation}
for the DRA, and 
\begin{equation}
d_B(z_n) < \ve 
\end{equation}
for the MAP and the MRP. 
From top to bottom in Figure~\ref{fg:orthant_tols}, 
we check how many iterates are required to get to tolerance 
$\ve =10^{-m}$, where $m \in \{2, 4\}$, respectively.
\begin{figure}[!ht]
\centering
\includegraphics[width =0.95\columnwidth]{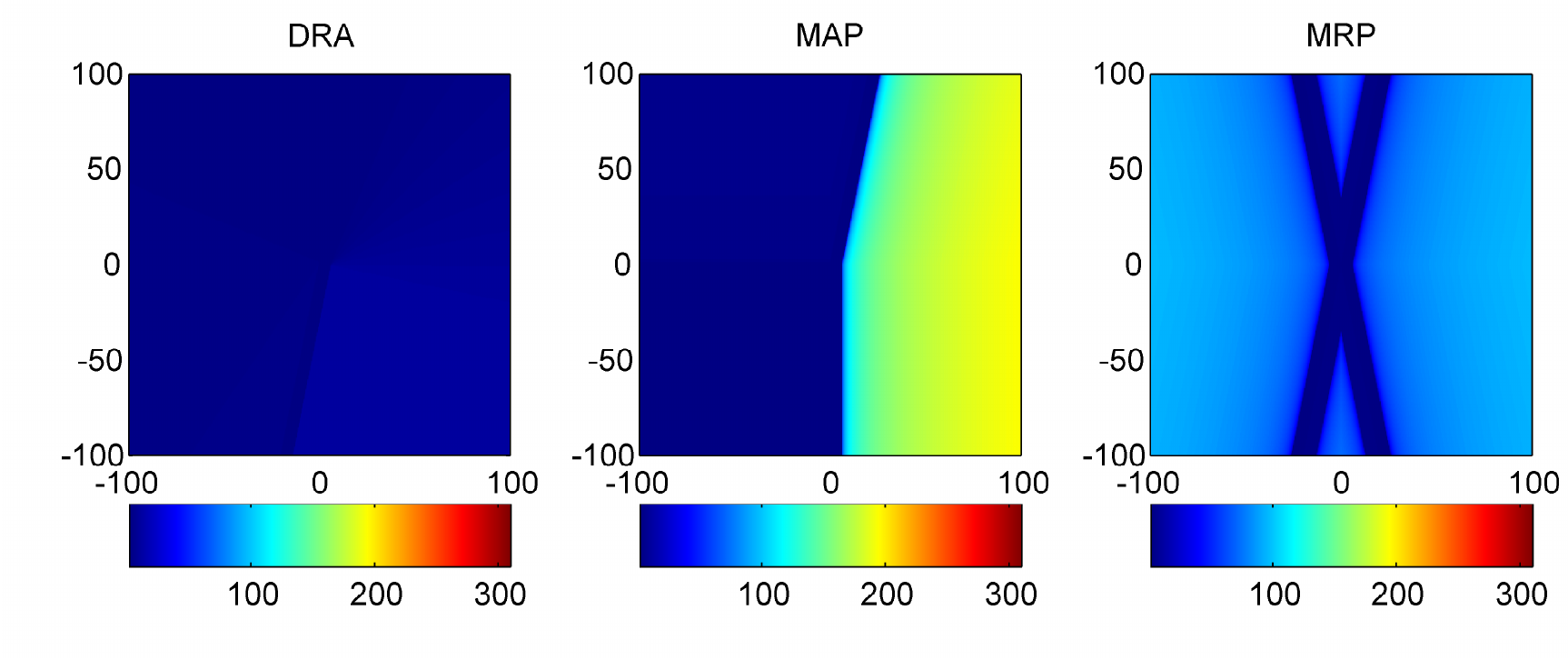}\\
\includegraphics[width =0.95\columnwidth]{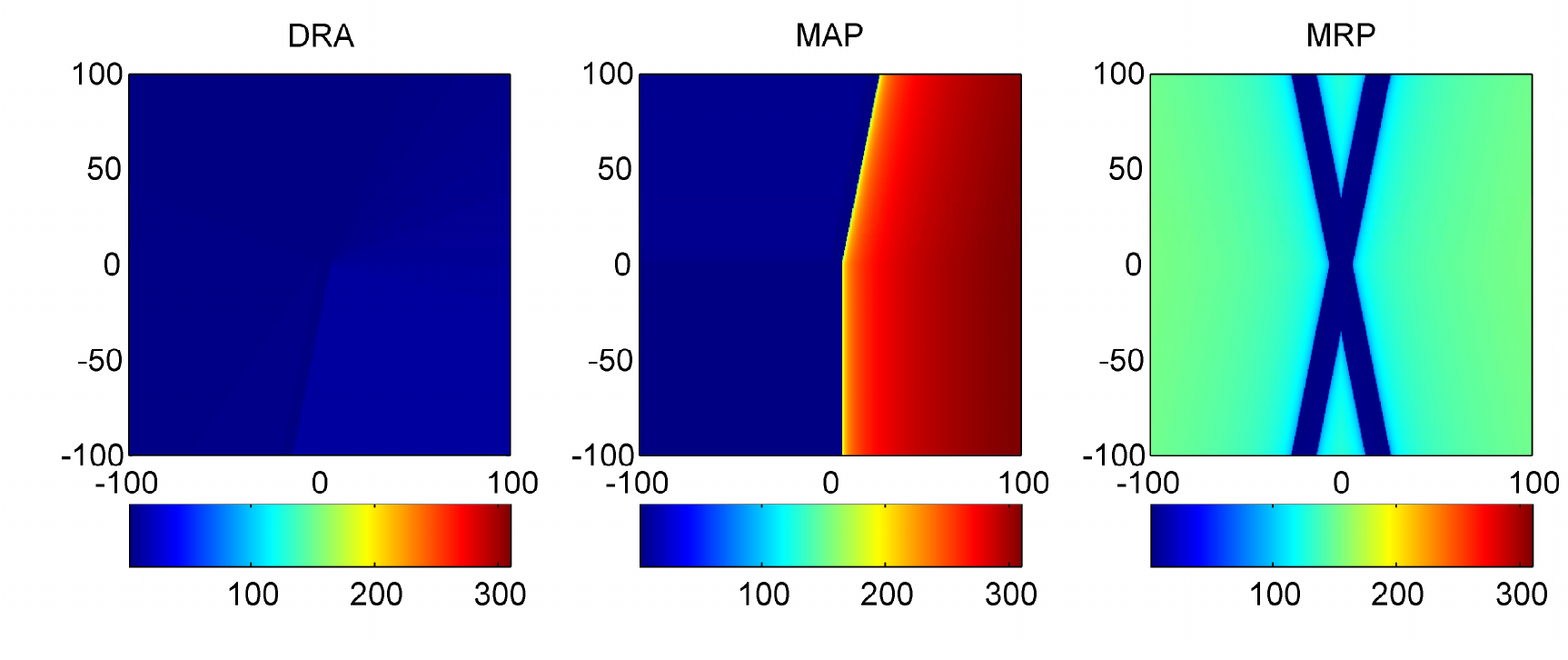}
\caption{Number of iterations needed to get the tolerance $10^{-2}$ (\emph{top}) and $10^{-4}$ (\emph{bottom}) of DRA, MAP and MRP.}
\label{fg:orthant_tols}
\end{figure}
Computations were performed with MATLAB R2013b \cite{Matlab}.
These experiments illustrate the superior convergence behaviour
of the DRA compared to the MAP and the MRP.

\section{The hyperplanar-epigraphical case with Slater's condition}

\label{s:hypepi}

In this section we assume that 
\begin{subequations}
\begin{empheq}[box=\mybluebox]{equation}
\text{$f \colon X \to \RR$ is convex and continuous.}
\end{empheq}
We will work in $X\times\RR$, where we set 
\begin{empheq}[box=\mybluebox]{equation}
A :=X\times \{0\}
\end{empheq}
and 
\begin{empheq}[box=\mybluebox]{equation}
B :=\epi f :=\menge{(x, \rho) \in X\times \RR}{f(x)\leq\rho}.
\end{empheq}
\end{subequations}
Then 
\begin{equation}
P_A \colon X\times\RR\to X\times\RR\colon
(x, \rho) \mapsto (x, 0),
\end{equation}
and the projection onto $B$ is described in the following result.

\begin{lemma}
\label{l:epi}
Let $(x, \rho) \in (X\times \RR)\smallsetminus B$. 
Then there exists $p\in X$ such that $P_B(x, \rho) =(p, f(p))$, 
\begin{equation}
\label{el:epi_proj}
x \in p +\big(f(p) -\rho\big)\partial f(p)\;\text{and}\; \rho <f(p) \leq
f(x)
\end{equation}
and 
\begin{equation}
\label{e:epi_proj}
(\forall y \in X)\quad \scal{y -p}{x -p} \leq \big(f(y)
-f(p)\big)\big(f(p) -\rho\big).
\end{equation}
Moreover, the following hold:
\begin{enumerate}
\item 
\label{l:epi_bdry}
If $u$ is a minimizer of $f$, then $\scal{u -p}{x -p} \leq 0$.
\item 
\label{l:epi_min}
If $x$ is a minimizer of $f$, then $p =x$.
\item 
\label{l:epi_strict}
If $x$ is not a minimizer of $f$, then $f(p) <f(x)$.
\end{enumerate}
\end{lemma}
\begin{proof}
According to \cite[Proposition~28.28]{BC11}, 
there exists $(p,p^*)\in\gra \partial f$ such that 
$P_B(x, \rho) =(p, f(p))$ and 
$x = p +(f(p) -\rho)p^*$.
Next, \cite[Propositions~9.18]{BC11} implies that 
$\rho <f(p)$ and \eqref{e:epi_proj}.
The subgradient inequality gives
\begin{equation}
\label{e:fxfp}
f(x) -f(p) \geq \scal{p^*}{x -p} =\scal{p^*}{(f(p) -\rho)p^*} =(f(p) -\rho)\|p^*\|^2 \geq 0.
\end{equation}
Hence $f(p) \leq f(x)$, and this completes the proof of \eqref{el:epi_proj}, 

\ref{l:epi_bdry}: It follows from \eqref{e:epi_proj} that
$\scal{u -p}{x -p} \leq (f(u) -f(p))(f(p) -\rho)$.
Since $f(p) -\rho >0$ and $f(u) \leq f(p)$, 
we have $\scal{u -p}{x -p} \leq 0$.

\ref{l:epi_min}: Apply \ref{l:epi_bdry} with $u =x$.

\ref{l:epi_strict}: By \eqref{el:epi_proj}, $\rho < f(p) \leq f(x)$. 
We show the contrapositive and thus 
assume that $f(p) =f(x)$. Since $f(p) -\rho >0$, 
\eqref{e:fxfp} yields $p^* =0$, and so $0 \in \partial f(p)$. 
It follows that $p$ is a minimizer of $f$, and therefore $x$ also 
minimizes $f$. 
\end{proof}

\begin{remark}
\label{r:pin}
When $X =\RR$ and $u$ is a minimizer of $f$, then 
$(u -p)(x -p) \leq 0$ by 
Lemma~\ref{l:epi}\ref{l:epi_bdry}.
Therefore, $p$ lies between $x$ and $u$ 
(see also \cite[Corollary~4.2]{BDNP15}).
\end{remark}

Define the DRA operator by 
\begin{empheq}[box=\mybluebox]{equation}
T :=\Id -P_A +P_BR_A. 
\end{empheq}
It will be convenient to abbreviate 
\begin{empheq}[box=\mybluebox]{equation}
\label{e:B'}
B' :=R_A(B) =\menge{(x, \rho) \in X\times \RR}{\rho \leq -f(x)}
\end{empheq}
and to analyze the effect of performing one DRA step in the
following result.

\begin{corollary}[one DRA step]
\label{c:DRstep}
Let $z =(x, \rho) \in X\times \RR$, 
and set $z_+ :=(x_+, \rho_+) =T(x, \rho)$. 
Then the following hold:
\begin{enumerate}
\item 
\label{c:DRstep_B'}
Suppose that $z \in B'$. 
Then $z_+ =(x, 0) \in A$.
Moreover, either 
($f(x)\leq 0$ and $z_+ \in A\cap B$) or  
($f(x) > 0$ and $z_+ \not\in B\cup B'$). 
\item 
\label{c:DRstep_oB'}
Suppose that $z \not\in B'$. 
Then there exists $x_+^* \in \partial f(x_+)$ such that 
\begin{equation}
\label{e:DRstep}
x_+ =x -\rho_+x_+^*,\; f(x_+) \leq f(x), \;\text{and}\;
\rho_+ =\rho +f(x_+) >0.
\end{equation}
Moreover, either ($\rho\geq 0$ and $z_+ \in B$)  
or ($\rho<0$, $z_+ \not\in B\cup B'$ and $Tz_+ \in B$). 
\item 
\label{c:DRstep_BB'}
Suppose that $z \in B\cap B'$. Then $z_+ \in A\cap B$.
\item 
\label{c:DRstep_BoB'}
Suppose that $z \in B\smallsetminus B'$. 
Then $z_+ \in B$. 
\end{enumerate}
\end{corollary}
\begin{proof}
\ref{c:DRstep_B'}: 
We have $P_Az =(x, 0)$ and $R_Az =(x, -\rho) \in B$. 
Thus $P_BR_Az =P_B(x, -\rho) =(x, -\rho)$, which gives
\begin{equation}
z_+ =(\Id -P_A +P_BR_A)z =(x, \rho) -(x, 0) +(x, -\rho) =(x, 0) \in A.
\end{equation}
If $f(x) \leq 0$, then $z_+ =(x, 0) \in B$, and hence $z_+ \in A\cap B$.
Otherwise, $f(x) >0$ which implies $-f(x) <0$ and further 
$z_+ =(x, 0) \not\in B\cup B'$.

\ref{c:DRstep_oB'}: We have $R_Az =(x, -\rho) \not\in B$, 
and by \eqref{el:epi_proj}, $P_B(x, -\rho) =(p, f(p))$, where
\begin{equation}
p =x -(\rho +f(p))p^* \text{ for some } p^* \in \partial f(p), \quad\text{and}\quad -\rho <f(p) \leq f(x).
\end{equation} 
We obtain
\begin{equation}
z_+ =(x_+, \rho_+) =(\Id -P_A +P_BR_A)z =(x, \rho) -(x, 0) +(p, f(p)) =(p, \rho +f(p)),
\end{equation}
which gives $x_+ =p$ and $\rho_+ =\rho +f(p) >0$, so \eqref{e:DRstep} holds.
If $\rho \geq 0$, then $\rho_+ =\rho +f(x_+) \geq f(x_+)$, and
thus $z_+ \in B$. Otherwise, $\rho <0$, so $\rho_+ <f(x_+)$
and also $f(x_+) >-\rho >0$. Hence $\rho_+ >0 >-f(x_+)$, 
which implies $z_+ \not\in B\cup B'$,
and then $Tz_+ \in B$ because $\rho_+ >0$ and the previous case
applies. 

\ref{c:DRstep_BB'}: 
We have $f(x) \leq \rho \leq -f(x)$, and so $f(x) \leq 0$. 
Now apply \ref{c:DRstep_B'}.

\ref{c:DRstep_BoB'}: We have $\rho \geq f(x)$ and $\rho >-f(x)$. 
Then $\rho \geq |f(x)| \geq 0$, and \ref{c:DRstep_oB'} gives $z_+ \in B$.  
\end{proof}

\begin{theorem}
\label{t:epi}
Suppose that 
$\inf_X f <0$, and, given a 
starting point $z_0 =(x_0, \rho_0) \in X\times \RR$,
generate the DRA sequence $(z_n)_\nnn$ by 
\begin{equation}
(\forall\nnn)\quad z_{n+1} =(x_{n+1}, \rho_{n+1}) =Tz_n.
\end{equation}
Then $(z_n)_\nnn$ converges finitely to a point $z \in A\cap B$.
\end{theorem}
\begin{proof}
In view of
Corollary~\ref{c:DRstep}\ref{c:DRstep_B'}\&\ref{c:DRstep_BB'},
we can and do assume that  $z_0 \in B\smallsetminus B'$,
where $B'$ was defined in \eqref{e:B'}. 
It follows then from 
Corollary~\ref{c:DRstep}\ref{c:DRstep_BB'}\&\ref{c:DRstep_BoB'},
that $(z_n)_\nnn$ lies in $B$.

\emph{Case 1}: $(\exists n \in \NN)$ $z_{n}\in B\cap B'$.\\
By Corollary~\ref{c:DRstep}\ref{c:DRstep_BB'},  
$z_{n+1}\in A\cap B$ and we are done.

\emph{Case 2}:  $(\forall\nnn)$ $z_n \in B\smallsetminus B'$.\\
By Corollary~\ref{c:DRstep}\ref{c:DRstep_oB'},
\begin{equation}
\label{e:xxrr}
(\forall\nnn)\quad f(x_{n+1}) \leq f(x_n)
\;\;\text{and}\;\; \rho_{n+1} =\rho_n +f(x_{n+1}) >0.
\end{equation}
Next, it follows from \cite[Lemma~7.3]{Roc70} that
$\inte B =\menge{(x, \rho) \in X\times \RR}{f(x)<\rho}$.
Because $\inf_X f <0$, we obtain $A\cap \inte B \neq\varnothing$,
which, due to Lemma~\ref{l:cvg}\ref{l:cvg_gen} yields
$z_n \to z =(x, \rho) \in A\cap B$.
Since $z \in A$, we must have $\rho =0$. 
If $(\forall\nnn)$ $f(x_{n+1}) \geq 0$, 
then, by \eqref{e:xxrr}, 
$0<\rho_1\leq\rho_2\leq\cdots\leq\rho_n\to \rho = 0$ which is
absurd. 
Therefore, 
\begin{equation}
(\exi n_0\in\NN)\quad f(x_{n_0+1}) <0.
\end{equation}
In view of \eqref{e:xxrr}, we see that 
$(\forall n \geq n_0+1)$
$\rho_n \leq \rho_{n_0} +(n -n_0)f(x_{n_0+1})$.
Since $f(x_{n_0+1}) <0$, there exists $n_1 \in \NN$, $n_1 \geq n_0 +1$ 
such that
\begin{equation}
\rho_{n_0} +(n_1 -n_0)f(x_{n_0+1}) \leq -f(x_{n_0+1}).
\end{equation} 
Noting that $f(x_{n_1}) \leq f(x_{n_0+1})$, we then obtain 
\begin{equation}
\rho_{n_1} \leq \rho_{n_0} +(n_1 -n_0)f(x_{n_0+1}) \leq
-f(x_{n_0+1}) \leq -f(x_{n_1}).
\end{equation}
Hence $z_{n_1} \in B'$, which contradicts the assumption of 
\emph{Case~2}. Therefore, \emph{Case~2} never occurs and
the proof is complete. 
\end{proof}

We conclude by illustrating that finite convergence may be
deduced from Theorem~\ref{t:epi} but not necessarily from
the finite convergence conditions of Section~\ref{s:Luque}. 

\begin{example}
\label{ex:epi}
Suppose that $X =\RR$, that $A =\RR\times \{0\}$, and that $B =\epi f$, 
where $f \colon \RR \to \RR\colon x\mapsto x^2 -1$.
Let $(\forall\ve\in\RPP)$
$z_\varepsilon = (1 +\varepsilon, -\varepsilon)$.
Then $(\forall\ve\in\RPP)$ $Tz_\varepsilon \notin \Fix T =[-1,1]\times\{0\}$, 
and $z_\varepsilon-Tz_\varepsilon \to 0$ as $\varepsilon \to 0^+$.
Consequently, Luque's condition \eqref{e:Lres} fails.
\end{example}
\begin{proof}
Let $\varepsilon \in \RPP$. Then $-f(1 +\varepsilon) =-2\varepsilon -\varepsilon^2 <-\varepsilon$,
and so $z_\varepsilon =(1 +\varepsilon, -\varepsilon) \notin B'$. 
By Corollary~\ref{c:DRstep}\ref{c:DRstep_oB'}, there exists $x_\varepsilon \in \RR$ such that 
\begin{equation}
Tz_\varepsilon =(x_\varepsilon, -\varepsilon +f(x_\varepsilon)),\quad  
x_\varepsilon =1 +\varepsilon -(-\varepsilon +f(x_\varepsilon))2x_\varepsilon, 
\quad\text{and}\quad -\varepsilon +f(x_\varepsilon) >0.
\end{equation}
The last inequality shows that $Tz_\varepsilon \notin \Fix T =[-1,1]\times\{0\}$.
It follows from the expression of $x_\varepsilon$ that 
\begin{equation}
\label{e:0421a}
2x_\varepsilon^3 -(1 +2\varepsilon)x_\varepsilon -1 -\varepsilon =0.
\end{equation}
Note that 
$(x_\ve,f(x_\ve))+ (0,-\ve) = Tz_\ve = z_\ve -P_Az_\ve  + P_BR_Az_\ve
= (1+\ve,-\ve)-(1+\ve,0)+P_B(1+\ve,\ve) =
P_B(1+\ve,\ve)+(0,-\ve)$ and hence 
$(x_\varepsilon, f(x_\varepsilon)) =P_B(1 +\varepsilon,
\varepsilon)$.
Remark~\ref{r:pin} gives $0 \leq x_\varepsilon \leq 1 +\varepsilon$.
If $0 \leq x_\varepsilon \leq 1$, then \eqref{e:0421a} yields 
$0 \leq 2x_\varepsilon -(1 +2\varepsilon)x_\varepsilon -1 -\varepsilon 
=(x_\varepsilon -1) -2\varepsilon x_\varepsilon -\varepsilon <0$,
which is absurd. 
Thus $1 <x_\varepsilon \leq 1 +\varepsilon$. 
Now as $\varepsilon \to 0^+$, we have $x_\varepsilon \to 1$, $f(x_\varepsilon) \to f(1) =0$, 
and at the same time, $z_\varepsilon-Tz_\varepsilon =(1 +\varepsilon -x_\varepsilon, -f(x_\varepsilon)) \to 0$.
\end{proof}

\small 

\subsection*{Acknowledgments}
HHB was partially supported by the Natural Sciences and
Engineering Research Council of Canada 
and by the Canada Research Chair Program.
MND was partially supported by an NSERC accelerator grant of HHB.

\end{document}